\documentclass{article}
\usepackage{amsmath,amssymb,amscd, float, verbatim, latexsym,cite}
\usepackage{amsthm}
\usepackage{graphicx}
\usepackage[all]{xy}
\usepackage{lscape,subcaption}
\usepackage{color,enumerate,setspace,url}		
\usepackage{bm}
\usepackage[margin=1.in]{geometry}
\usepackage{mathtools}
\mathtoolsset{showonlyrefs=true}

\newtheorem{theorem}{Theorem}[section]
\newtheorem{lemma}[theorem]{Lemma}
\newtheorem{conjecture}[theorem]{Conjecture}

\newtheorem{remark}{Remark}[section]
\newtheorem{definition}[theorem]{Definition}

\newcommand{\C}{\mathbb{C}}

\newcommand{\N}{\mathbb{N}}
\newcommand{\R}{\mathbb{R}}

\newcommand{\Z}{\mathbb{Z}}

\newcommand{\expig}{e^{i \theta}}
\newcommand{\bydef}{\,\stackrel{\mbox{\tiny\textnormal{\raisebox{0ex}[0ex][0ex]{def}}}}{=}\,}

\newcommand{\ta}{\tilde{a}}
\newcommand{\tp}{\tilde{p}}
\newcommand{\tb}{\tilde{b}}
\newcommand{\ba}{\bar{a}}

\newcommand{\bp}{\bar{p}}

\newcommand{\cB}{\mathcal{B}}

\newcommand{\cS}{\mathcal{S}}

\title{
Singularities and heteroclinic connections in complex-valued evolutionary equations with a quadratic nonlinearity
}
\author{%
Jonathan~Jaquette\thanks{%
	Department of Mathematics and Statistics, Boston University,
	Boston, MA 02215, USA. 
}
\and
Jean-Philippe~Lessard\thanks{%
	Department of Mathematics and Statistics, McGill University,
	Montreal, QC H3A 0B9, Canada
}\and
Akitoshi~Takayasu\thanks{%
	Faculty of Engineering, Information and Systems, University of Tsukuba, 1-1-1 Tennodai, Tsukuba, Ibaraki 305-8573, Japan (\texttt{takitoshi@risk.tsukuba.ac.jp})
}
}

\begin{document}

\maketitle

\begin{abstract} 
In this paper, we consider the dynamics of solutions to complex-valued evolutionary partial differential equations (PDEs) and show existence of heteroclinic orbits from nontrivial equilibria to zero via computer-assisted proofs. We also show that the existence of unbounded solutions along unstable manifolds at the equilibrium follows from the existence of heteroclinic orbits. Our computer-assisted proof consists of three separate techniques of rigorous numerics: an enclosure of a local unstable manifold at the equilibria, a rigorous integration of PDEs, and a constructive validation of a trapping region around the zero equilibrium. 
\end{abstract}

\noindent
{\bf Keywords:} nonlinear heat equation, heteroclinic connections, global existence of solution,  rigorous numerics


\section{Introduction} \label{sec:Introduction}


Understanding the long term behavior of solutions to evolutionary equations is fundamental to the discipline.  
This begins with questions of existence, and  whether local  existence can be extended globally. 
Of the solutions which exist globally, coherent structures (such as equilibria, traveling waves and periodic orbits) serve as emblematic examples of how solutions to a PDE may behave. 
The local stability of these coherent structures helps to inform us as to what type of phenomena are generically observable: while stable structures are robust and attract nearby trajectories, unstable objects repel solutions and are harder to detect. 
Nevertheless unstable objects are critically important in understanding the transient behavior of a system, and  understanding to which particular stable state solutions will be attracted. 

Characterizing the long term behavior of trajectories near an unstable object requires a global analysis, combining a perturbative analysis of solutions near the unstable object, with a non-perturbative analysis for when the solution is far from the invariant object. 
The non-perturbative analysis is genuinely difficult, and one often turns to numerical methods for insight. 
There are well established algorithms for integrating PDEs within a certain degree of precision. 
However in cases where the global existence of solutions is unknown, such as the 3D Navier-Stokes, it is not always obvious whether numerical results indicative of blowup are a genuine feature of the dynamics, or an artifact of the numerical method's finite precision.

It is along these lines that we study the  global dynamics and global wellposedness (or lack thereof) of the class of  complex-valued nonlinear heat equations  
\begin{equation} \label{eq:CGL}
	u_t = \expig (u_{xx} + u^2), \quad \theta \in [-\pi/2,\pi/2]
\end{equation}
under the periodic boundary condition in $x\in(0,1)$, where $u_t=\partial_t u$ and $u_{xx}=\partial^2_x u$.  
As a case study, we investigate trajectories on the  unstable manifold of a nontrivial equilibrium of \eqref{eq:CGL} at parameters $ \theta \in \{ 0 , \pi /4 , \pi /2\}$. 
We numerically observe that most solutions are heteroclinic to the zero solution, however some solutions appear to grow without bound. 
By employing state-of-the-art techniques in validated numerics, we are able to rigorously characterize the limiting behavior of some solutions, and identify others where the numerical precision is lacking.  
Furthermore, through a forcing argument we can confirm the existence of unbounded solutions within the unstable manifold  when $ \theta \in \{ 0 , \pi /4 \}$.

Equation \eqref{eq:CGL} can be seen as a generalization of a classical model for blowup in PDEs.  
When $ \theta =0$ and the initial data is real, this reduces to the quadratic case of the Fujita equation $u_t = u_{xx} + u^{p}$  on a bounded domain \cite{fujita} whose blowup properties for non-negative initial data have been studied extensively \cite{levine1990role,deng2000role}.  However as a complex equation the condition of non-negative real initial data is not an open property and finite time blowup, while present, does not occur as robustly \cite{S0002-9947-2012-05797-7}.    
By varying $\theta$, we are able to compare the blowup phenomena and mechanisms between qualitatively disparate types of evolutionary equations. 
When $ \theta =0$ then \eqref{eq:CGL} is a nonlinear heat equation,  when $\theta = \pm \pi/ 2$ then \eqref{eq:CGL} is a nonlinear Schr\"odinger equation, and for other values of $\theta$ the equation is similar to the complex Ginzburg-Landau equation.

Furthermore  this equation arises from the mathematical studies by Masuda \cite{MASUDA1983119,Masuda1984} and numerical observations by Cho et al.\ \cite{Cho2016} to consider the nonlinear heat equation in the complex plane of time
\begin{equation}\label{eq:CF}
	u_{z}=u_{x x}+u^{2}
\end{equation}
on a straight path
$\Gamma_\theta\bydef\left\{z\in\C : z = t\expig\right\}$. 
This setting is slightly different from typical settings of complex-valued nonlinear heat equations, which often consider the complex-valued unknown function $u\in\C$ and both time and space variables are real-valued, see \cite{S0002-9947-2012-05797-7} for example. 
Our setting here is the case of nonlinear heat equations whose time variable is complex-valued.

As an innovative work, Masuda has considered the Cauchy problem of \eqref{eq:CF} under Neumann boundary conditions in \cite{MASUDA1983119,Masuda1984}. He has proved global well-posedness of analytic solutions of \eqref{eq:CF} in a specific domain and existence of branching singularities at the movable singularity, which is called \emph{blow-up times} in the case of real-valued nonlinear heat equations. Following Masuda's studies, Cho et al.\ have numerically observed solution dynamics of the Cauchy problem of \eqref{eq:CF} under the periodic boundary condition in \cite{Cho2016}. They have also presented that the solution may converge to the zero function on the straight path $\Gamma_\theta$. This result agrees with Masda's mathematical results in the case of the periodic boundary conditions.

Moreover, the present authors with H.\ Okamoto have studied the Cauchy problem of  \eqref{eq:CF} under the same setting as Cho et al.\ \cite{Cho2016} in \cite{takayasu2019rigorous} and have proved two results with computer-assistance. The first one is that there exists a branching singularity at the blow-up time, which extends the mathematical results by Masuda and mathematically proves numerical observations by Cho et al. The second one is global existence of the solution to the Cauchy problem of \eqref{eq:CGL} for some $\theta$. In particular, the initial data is given by $u_{0}(x)=50(1-\cos (2 \pi x))$, which is relatively large and far from a constant function. Therefore, this result agrees with Masuda's work without the assumption of closeness to a constant. The proofs of these results are obtained by the tools of \emph{rigorous numerics} such as rigorous integrator of time-dependent PDEs and numerical validation of center-stable manifolds. 
For further  work using complexification to regularize singularities see  \cite{kevrekidis2017infinity} and the  references therein.

In contrast to the strictly real version of \eqref{eq:CF}, the complex PDE in \eqref{eq:CGL} possesses  a diverse array of global solutions and a rich dynamical structure.  
In \cite{Jaquette2020} it was shown that when $ \theta = \pi/2$ this equation has non-trivial complex-valued equilibria, whose profiles are shown in Figure \ref{fig:equilibria}. 
One may immediately see that the equilibria of \eqref{eq:CGL} are independent of $\theta$. 
In \cite{Jaquette2020} it was also shown that there exists  an open set of homoclinic connections at the zero function, that is
solutions which limit to the zero in both forward and backward time. 
Additionally in \cite{jaquette2021quasiperiodicity} it is shown that for  $ \theta = - \pi /2$, there are periodic orbits, and large  monochromatic initial data which blows up in the $L^2$ norm.

\begin{figure}[htbp]
	\centering
	\includegraphics[width = .48 \textwidth]{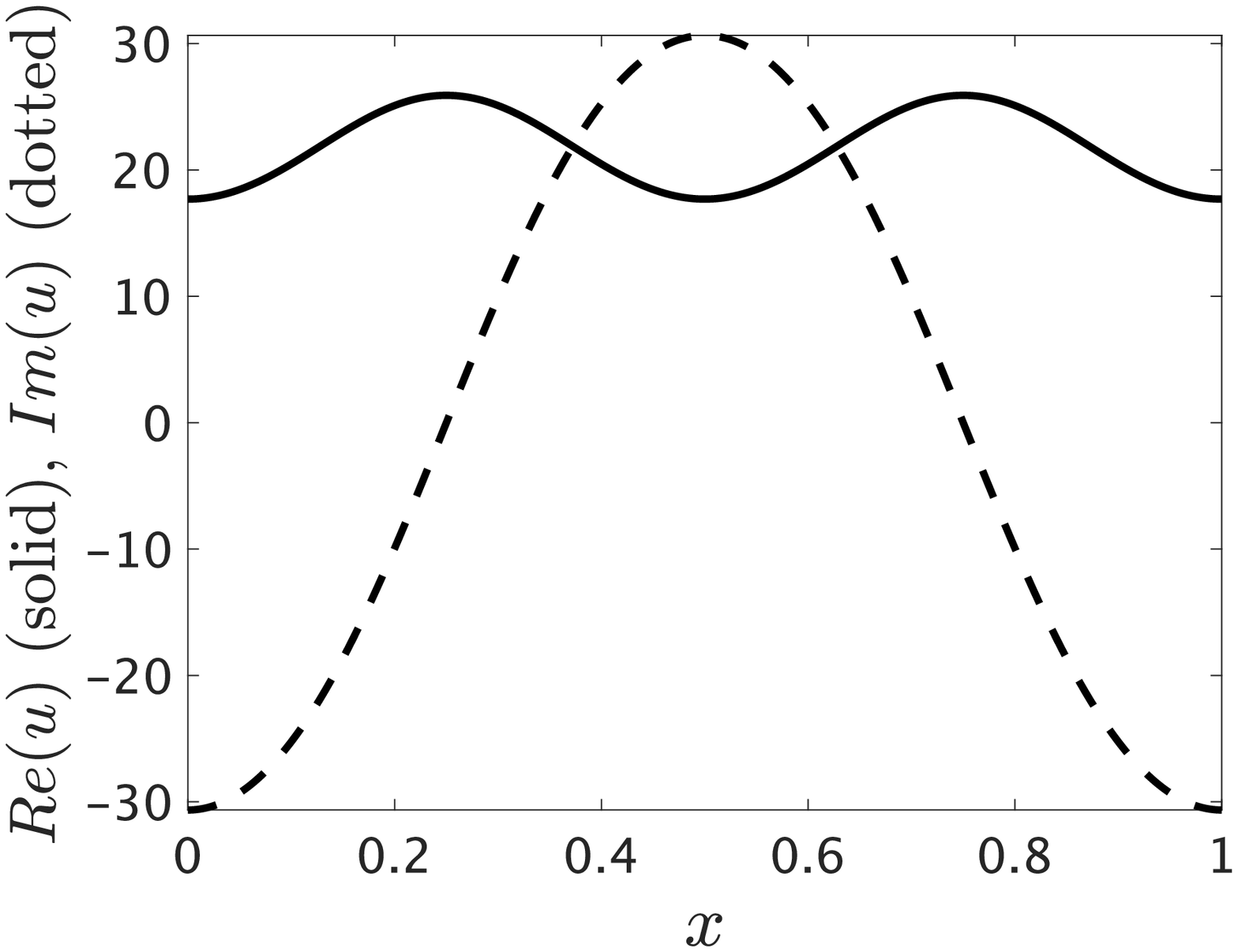}
	\includegraphics[width = .48 \textwidth]{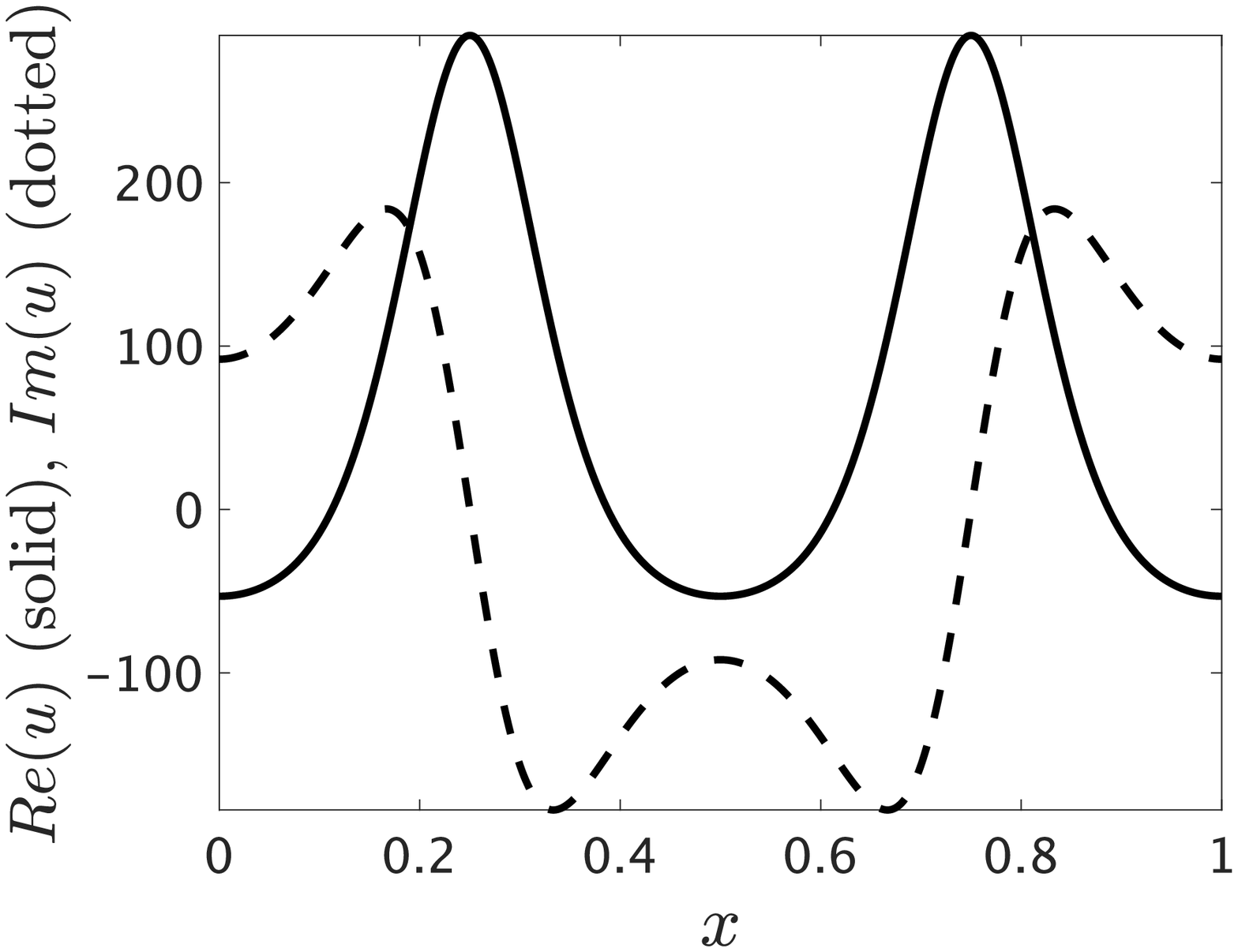}
	\caption{Two different nontrivial equilibria: $u_1(x)$ (left), $u_2(x)$ (right), the existence of which are established via computer-assisted proof \cite[Theorem 1.7]{Jaquette2020}. 
	}\label{fig:equilibria}
\end{figure}

Note also that since \eqref{eq:CGL} has a nonlinearity of homogeneous degree  then there exists a rescaling of solutions. 
That is, if $u(t,x)$ solves \eqref{eq:CGL} then
\[
\tilde{u}(t, x) \bydef n^{2} u(n^{2} t, n x), \quad \forall n \in \mathbb{Z}
\]
solves \eqref{eq:CGL} as well. 
Hence the nontrivial equilibria proved to exist in \cite{Jaquette2020} in fact generate infinite families of equilibria of arbitrarily large norm. 
Furthermore, if there exists a heteroclinic orbit between the base nontrivial equilibria and the zero function, then it yields existence of a countable infinite number of heteroclinic orbits.  
This is in fact the case when $\theta = \pi/2$.

\begin{theorem}[\!\!{\cite[Theorem 1.9]{Jaquette2020}}]\label{thm:Heteroclinics1}
	Consider the case of $\theta=\pi/2$ in \eqref{eq:CGL}.
	Let $u$ be any one of the equilibria constructed from $u_1(x)$ or $u_2(x)$ in Figure \ref{fig:equilibria}.	There exist heteroclinic orbits  $u_a$ and $u_b$ to \eqref{eq:CGL} such that 
	\begin{align*}
		\lim_{t \to -\infty} u_a(t) &= u,&
		\lim_{t \to +\infty} u_a(t) &= 0,
		& & &
		\lim_{t \to -\infty} u_b(t) &= 0,
		&
		\lim_{t \to +\infty} u_b(t) &=  u,
	\end{align*}
	converging exponentially fast to $u$ and algebraically fast to $0$. 
\end{theorem}

The heteroclinic orbit $u_a$ is obtained by  combining three separate techniques of computer-assisted proofs: an enclosure of a local unstable manifold at the equilibria, a rigorous integration of the flow (starting from the unstable manifold) and a proof that the solution enters a validated stable set.  
The heteroclinic orbit $u_b$ is then obtained as a corollary, relying on the fact  that when $ \theta = \pi /2 $ then \eqref{eq:CGL} is a nonlinear Schr\"odinger equation, and moreover a reversible equation. 
In the case  $\theta\in(-\pi/2,\pi/2)$ we would expect the heteroclinic orbit $u_a$, although probably not $ u_b$, to persist.




In this paper, we perform a systematic study of the long term behavior of trajectories on the unstable manifold of the equilibrium $u_1$ in Figure \ref{fig:equilibria} for $\theta \in \{0,\pi/4 ,\pi/2\}$.   
To focus our analysis we will restrict ourselves to studying a 1-complex dimensional submanifold. 
We will note however that while $ u_1$ is unstable for all $\theta \in [-\pi/2,\pi/2]$, its Morse index (that is the number of eigenvalues with positive real part), and hence the dimension of its unstable manifold does not stay constant as $\theta$ varies.

Indeed, for the equilibrium solution $u_1$ to \eqref{eq:CGL}, the eigenvalue problem is to find a pair $ ( \tilde{\lambda}, \tilde{h})$ such that
\begin{align}\label{eq:linearized_eigprob}
	e^{i \theta } \left(
	\triangle \tilde{h} + 2 u_1 \tilde{h} 
	\right)
	= \tilde{\lambda} \tilde{h}.
\end{align}
If there exists an eigenpair $(\tilde{\lambda},\tilde{h})$ for $ \theta_0= 0$, then $(e^{i \theta} \tilde{\lambda} , e^{i \psi} \tilde{h})$ will be an eigenpair for any $ \theta \neq 0$ and $ \psi \in [0, 2 \pi]$. 
For $\theta_0 = 0$, we numerically calculate the eigenvalues with the largest real part to be: 
\begin{align}
	\lambda_1 & \approx 17.696 + 35.391 i \\
	\lambda_2 &\approx 17.696 - 35.391 i \\
	\lambda_3 &\approx -1.1109\times10^{2} 
	\\
	\lambda_4 &\approx-3.1038\times10^{2}.
\end{align}
The rest of the eigenvalues appear to have zero imaginary part, see Figure \ref{fig:eigenvalues}, but this we cannot prove.

\begin{figure}[htbp]
	\centering
	\includegraphics[width = .8 \textwidth]{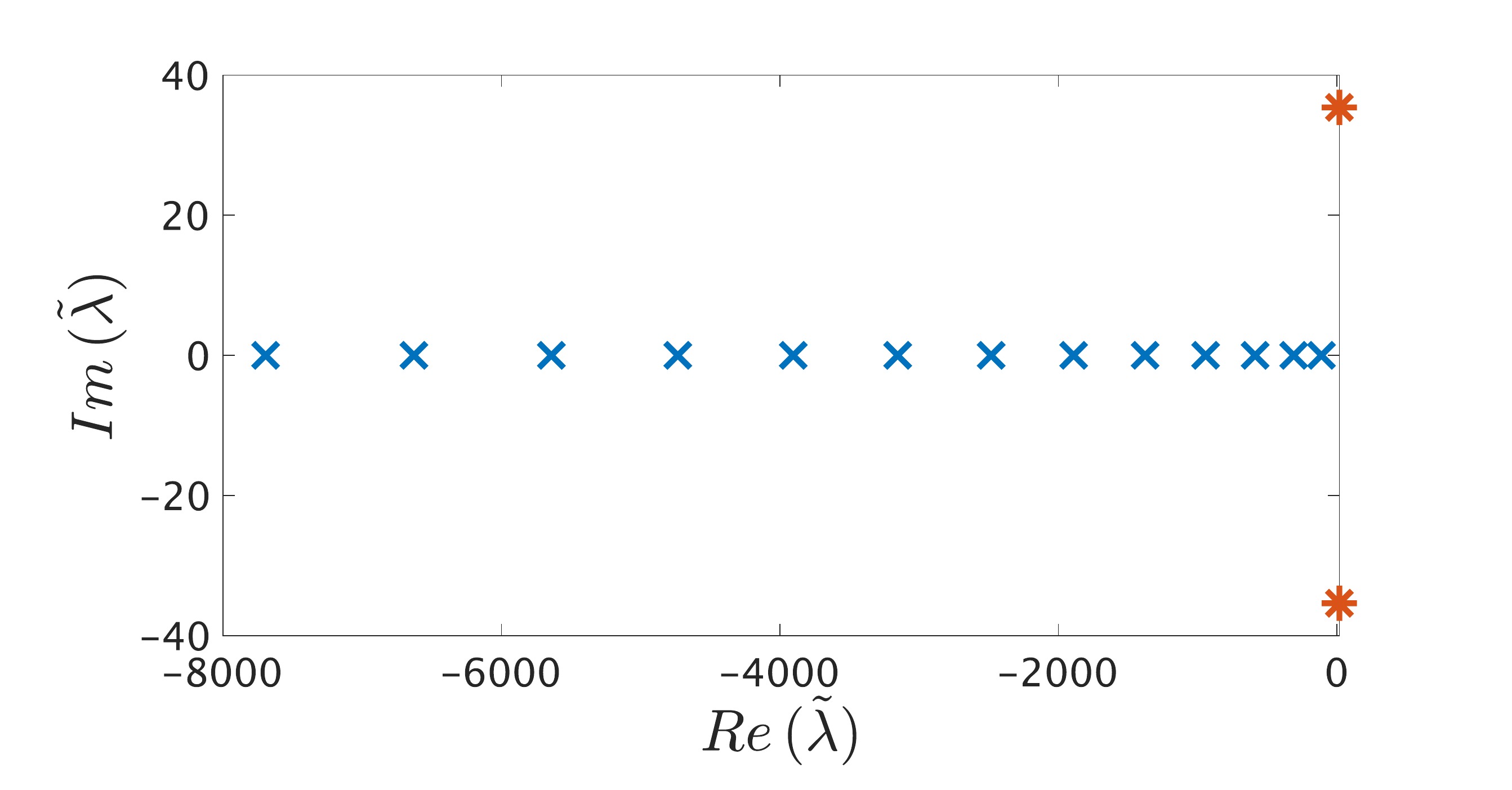}
	\caption{Numerically computed (non-rigorous) eigenvalues $\tilde{\lambda}$ of \eqref{eq:linearized_eigprob} for $ \theta = 0$. There are two unstable eigenvalues ${\lambda}_1$ and ${\lambda}_2$ (asterisk). The others (crosses) are stable and located on the real axis.}\label{fig:eigenvalues}
\end{figure}

Given the rescaling of eigenvalues by the parameter $\theta$ and our calculation of the eigenvalues when $ \theta = 0$, we can numerically determine the Morse index of the equilibrium. 
Then at some $\theta = \theta^* \approx 0.463$, there is one unstable eigenvalue, one purely imaginary eigenvalue, and infinitely many stable eigenvalues.
Furthermore, for $ |\theta| \in [0, \theta^*)$ this equilibrium has two unstable eigenvalues, and for   $ |\theta| \in ( \theta^*, \pi /2]$ the equilibrium has only one unstable eigenvalue. 
Hence, for all $ \theta \in [0, \pi/2]$ the equilibrium will have an unstable manifold of complex dimension of $1$ or $2$.


To restrict our analysis to something more tractable, we focus on the equilibrium $u_1(x)  $ shown in Figure~\ref{fig:equilibria}.
Then we will consider just the 1-complex dimensional unstable manifold of the equilibrium associated to the eigenvalue with the largest positive real part (sometimes called the {\em strong} unstable manifold).  
Let $\psi_k\bydef(2\pi/360)k$ be an indexed parameter by integer $k\in\N$.
We will study solution behavior coming out of the 1-complex dimensional unstable manifold associated with the eigenpair $(e^{i \theta} \tilde{\lambda} , e^{i \psi_k} \tilde{h})$.
The index $k$ is taken from $0$ to $360$ so that we go around the manifold to investigate the solution behavior.
Along this {\em circular} organization of the unstable manifold, we are able to show that many points will converge to zero. (These are just for particular angles. While the proof does extend to some open set about each angle, this is not the precise theorem we are stating here.) 
However for other angles we are not able to complete the proof. Below is what we can rigorously prove.

\begin{remark} \label{rem:fixing_the_eigenpair}
	Throughout the paper we work with a fixed eigenpair $(  e^{i \theta} \tilde{\lambda}, e^{i \psi_k} \tilde{h})$ such that $ Re(  e^{i \theta} \tilde{\lambda}) > 0$ for all $ \theta \in [ 0, \pi/2]$. 
	This selection of eigenvalue $\tilde{\lambda}$ corresponds with the lower red asterisk in Figure \ref{fig:eigenvalues}. 
	For the choice of eigenvector  $\tilde{h}$, which in this case is neither unique in its phase nor its amplitude, we have fixed a particular choice in our code  \cite{bib:codes}. 
\end{remark}

\begin{theorem} \label{prop:ManyHeteroclinics}
	Consider $u_1$ the nontrivial equilibrium to equation \eqref{eq:CGL} and fix the eigenpair as described in Remark~\ref{rem:fixing_the_eigenpair}. For $ \theta \in \{0,\pi/4,\pi/2\}$, let $ P_\theta$ be a coordinate chart for a 1-complex dimensional submanifold of the unstable manifold. 
	
	\begin{itemize}
		\item[(a)] For $ \theta = 0$ we prove that there exists a heteroclinic orbit from $u_1$ to $0$ (through the point $ P_\theta( e^{i \psi_k})$) for each integer  degree angle $ 0 \leq k \leq 360$ except for the range $321 \leq  k  \leq 357$. 
		\item[(b)] For $ \theta = \pi/4$ we prove that there exists a heteroclinic orbit from $u_1$ to $0$ (through the point $ P_\theta( e^{i \psi_k})$) for each integer  degree angle $ 0 \leq k \leq 360$ except for the range $311 \leq  k  \leq 332$.
		\item[(c)] For $ \theta = \pi/2$ we prove that there exists a heteroclinic orbit from $u_1$ to $0$ (through the point $ P_\theta( e^{i \psi_k})$) for each integer  degree angle $ 0 \leq k \leq 360$ except for the range $214 \leq  k  \leq 321$.
	\end{itemize}
	
\end{theorem}

\begin{figure}[htbp]
	\centering
	\hfill 
	\includegraphics[width = .45 \textwidth]{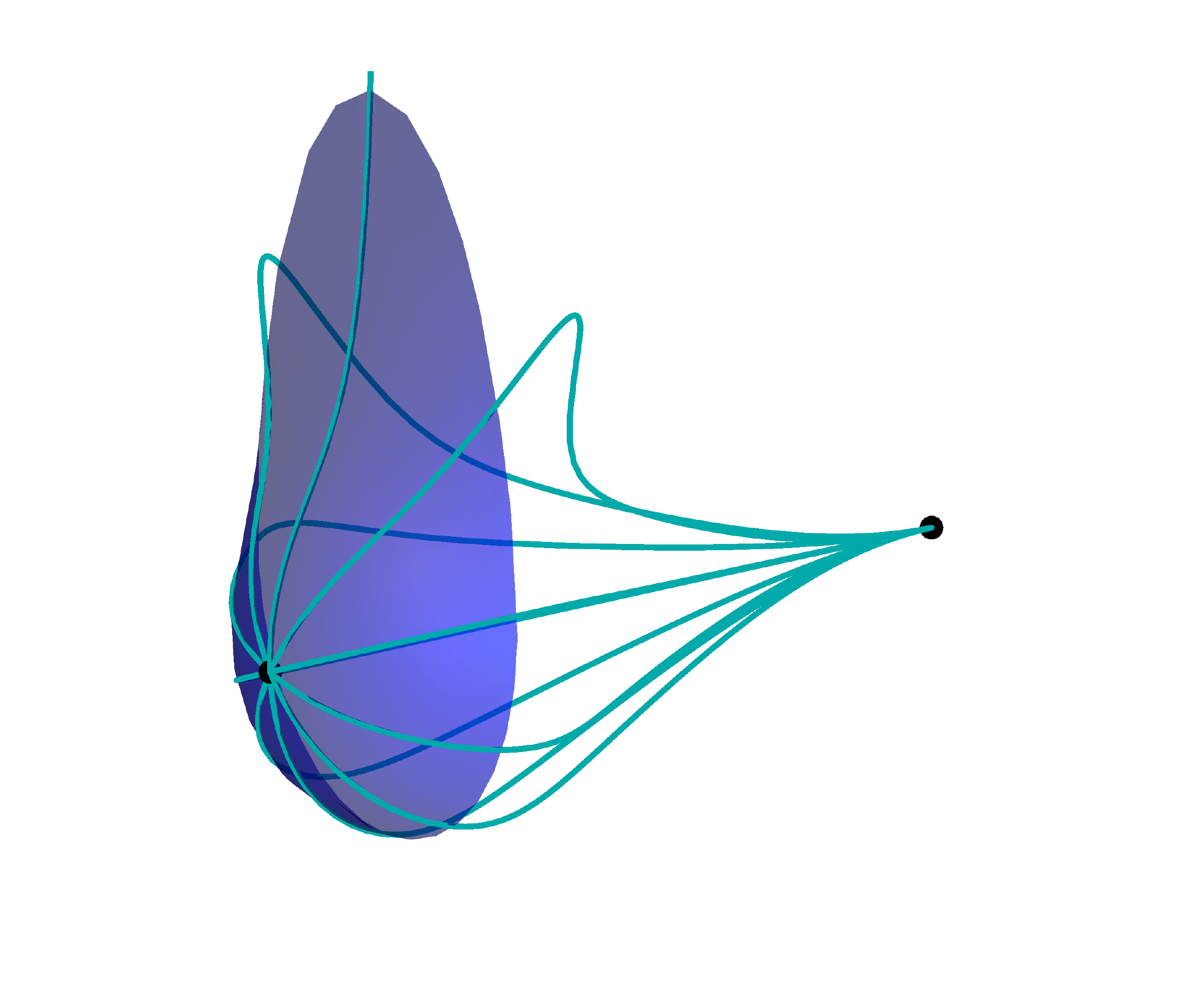} 
	\hfill 
	\includegraphics[width=.49\textwidth]{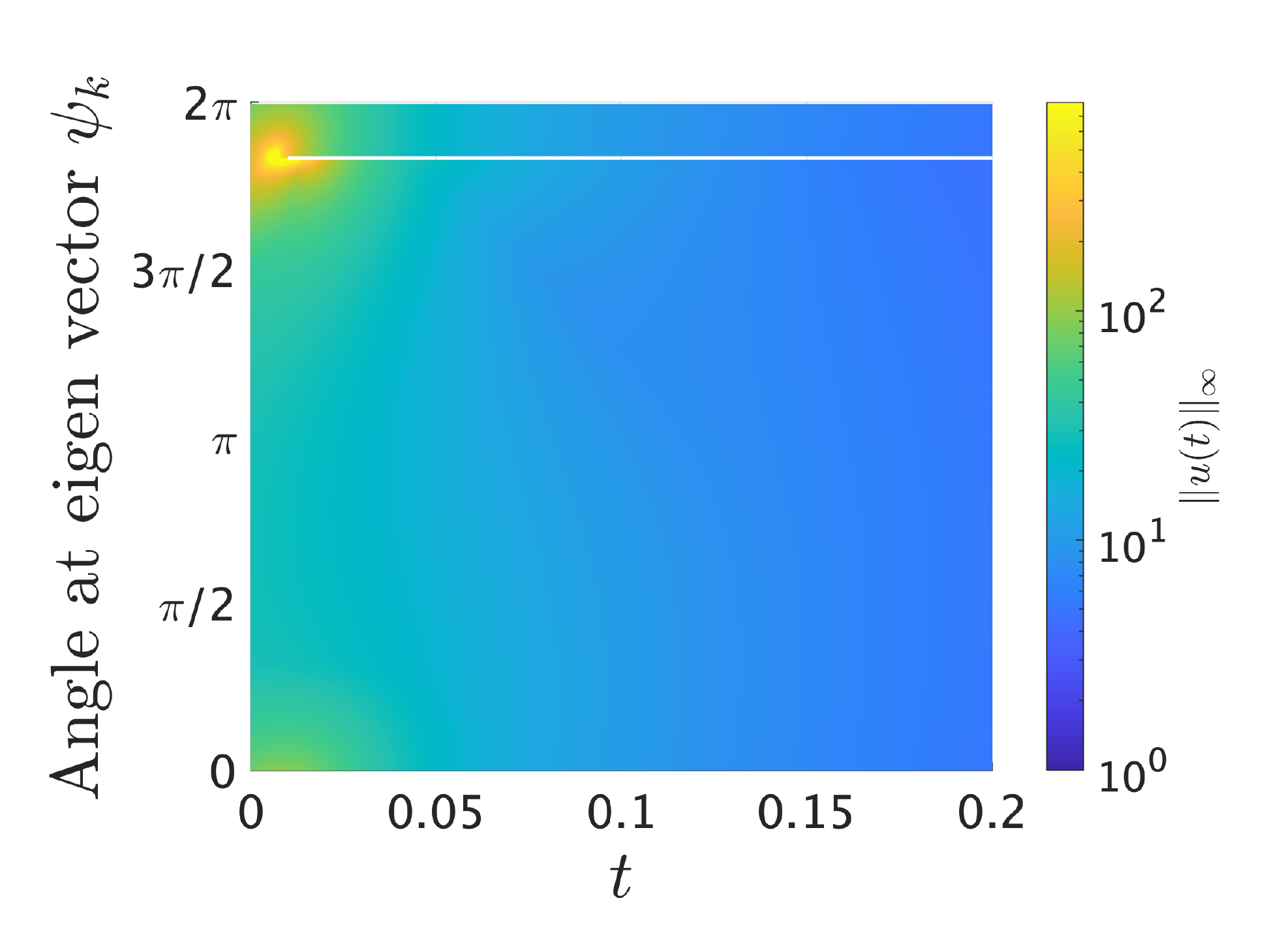}
	\hfill 
	\null
	
	\caption{(left) A cartoon model of the unstable manifold of a non-trivial equilibrium. While most trajectories on the manifold converge to the trivial equilibrium, it appears that for a particular angle coming out of the unstable manifold, that trajectory will blowup in finite time. (right) We plot values of time-dependent supremum norm of solutions on the unstable manifold. Varying the angles $\psi_k$, it seems that most trajectories on the manifold converges to zero except for some angles. The solution profile looks blowup numerically. As the cartoon on the left suggests, there could exist blow-up solution of \eqref{eq:CGL}.
	}\label{fig:onion}
\end{figure}

For  angles $\psi_k$ at which we are unable to obtain a computer-assisted proof (CAP), the long term behavior of the solution is unclear. One possibility is that the solution will eventually converge to zero, however the precision of our computation was insufficient to produce a CAP. Another possibility is that the solution will grow without bound, either in finite time (a blowup solution) or in infinite time (a growup solution). We suspect the angles where our CAP fails to form a neighborhood about a growup/blowup solution, see Figure \ref{fig:onion}. For example, in our purely numerical calculations at $ \theta = \pi/4$ our non-rigorous integrator failed at only a single angle $\psi_{330}$. Furthermore, solutions on either side of this {\em non-proof interval} display qualitatively different behavior suggestive of a branching singularity, see Figure \ref{fig:CO}. We note that our CAP approach succeeded in proving non-trivial solution behaviors for the angle $\psi_{333}$, whose amplitude is significantly large. 

\begin{figure}[htbp]
	\centering
	\includegraphics[width = .85 \textwidth]{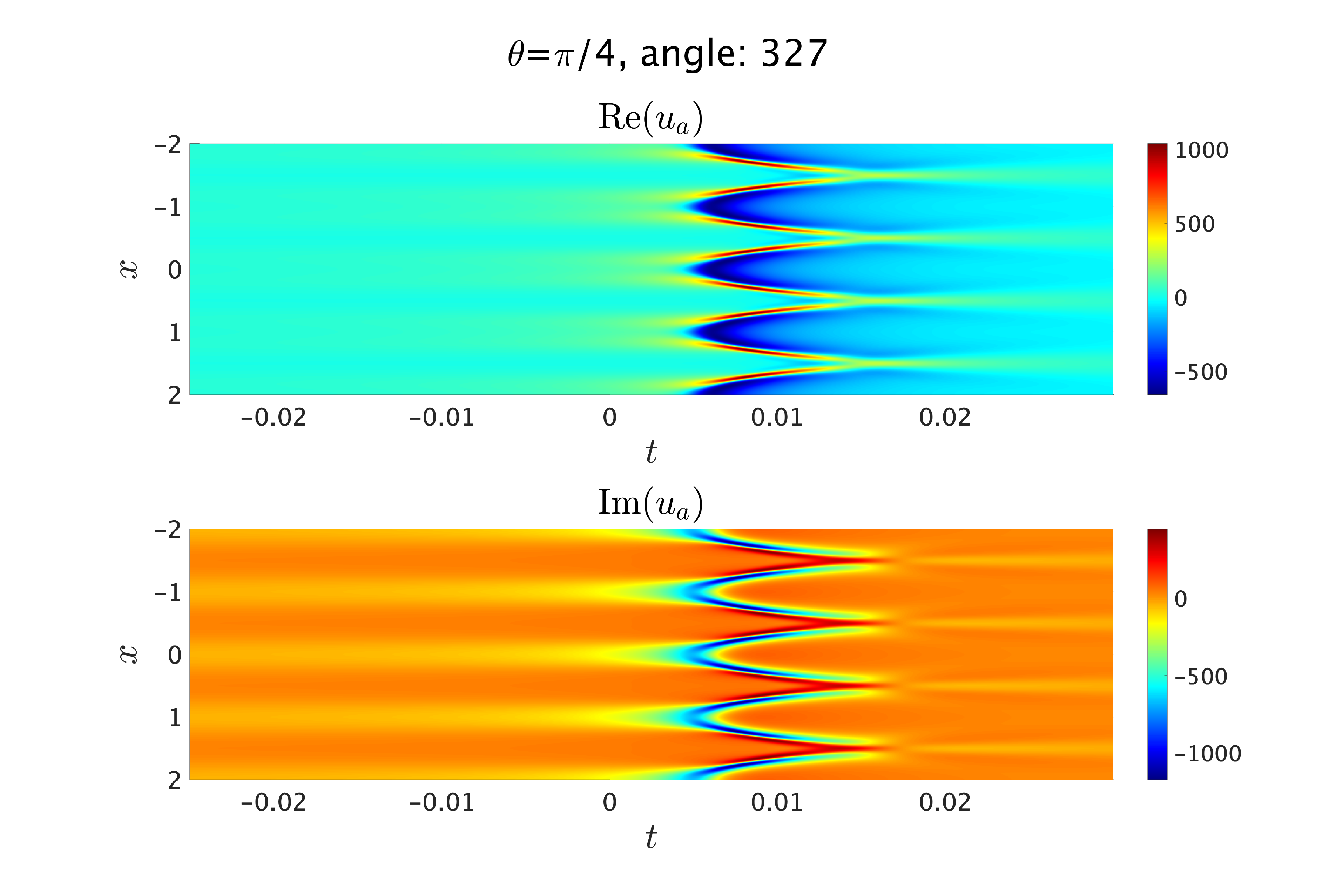}\\[-4pt]
	($a$) A part of heteroclinic connection from $u_1(x)$ to $0$ at the angle $\psi_k$ with $k=327$. 
	\includegraphics[width = .85 \textwidth]{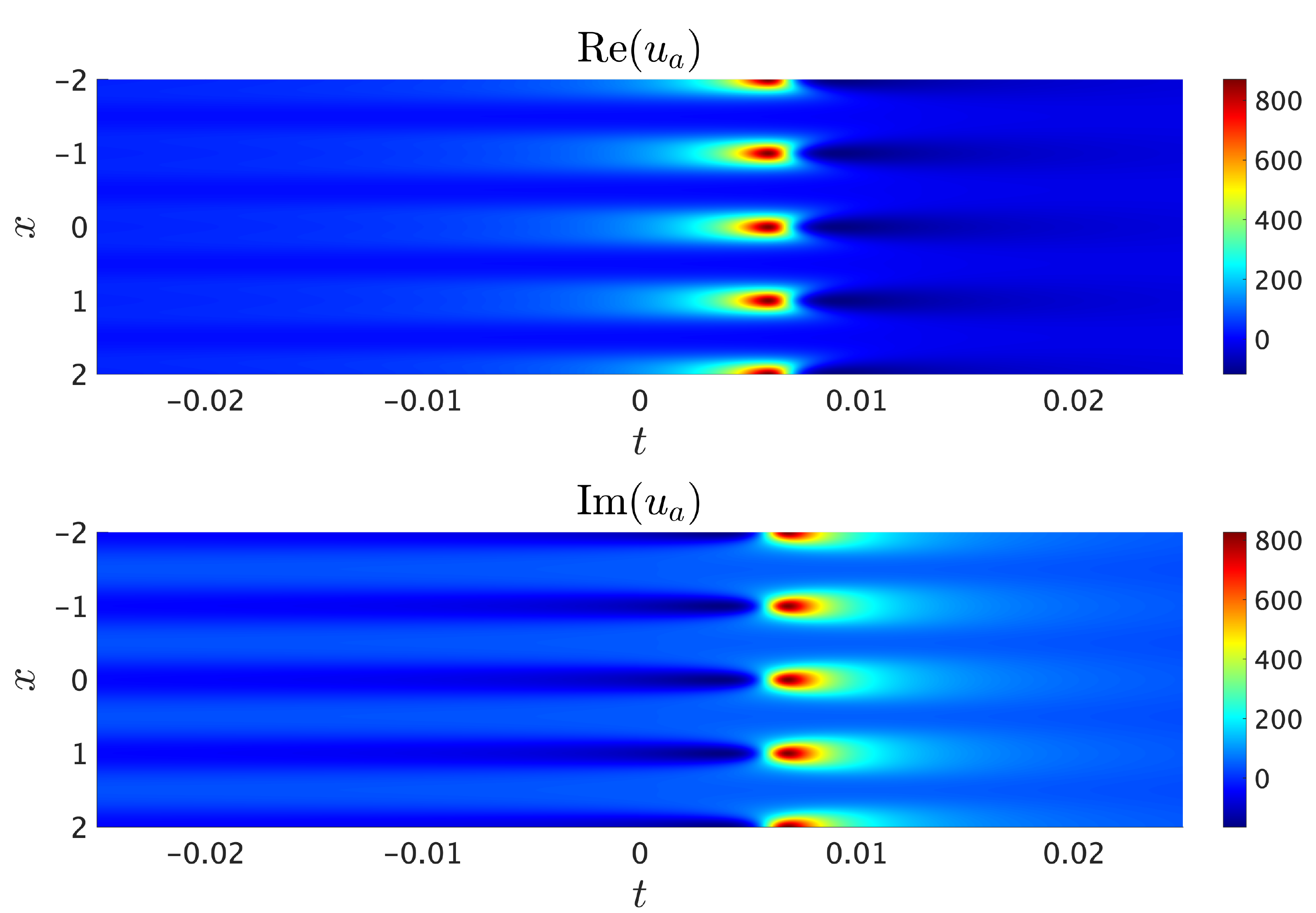}\\[-4pt]
	($b$) A part of heteroclinic connection from $u_1(x)$ to $0$ at the angle $\psi_k$ with $k=333$.
	\caption{Solution profile of heteroclinic orbits for $\theta=\pi/4$. (a) Profile for $\psi_{327}$, for which our computer-assisted proof fails. (b) Profile for $\psi_{333}$, for which we succeeded in proving the existence of the orbit. These solutions seem to behave quite differently. Here the orbit is zoomed in near $t=0$ for ease of viewing and we extend the $x$ variable to $x\in [-2,2]$ in order to clearly understand the time evolution of the solution.}\label{fig:CO}
\end{figure}

By using validated numerics we are able to be notified of the possible existence of singularities that would otherwise be subtle to detect.  
If we had used a coarser resolution of angles $\psi_k$ in our  calculation of Figure \ref{fig:onion} (right) we may have missed the single angle $ \psi_{330}$ where  our non-rigorous numerics break down. 
(In fact for $\theta =0$ our non-rigorous computations did not breakdown for any of the angles $ \psi = 0 ,1, \dots 359$.)  
While evidence of a branching singularity can still be found in a faint discontinuity in the graph of $ \|u(t)\|_{\infty}$ (as $t$ and $ \psi_k$ are varied), it is a subtle matter to verify. In contrast, by using validated numerics we are readily alerted to a possible singularity.   
However with the present methods it is unclear how to distinguish between an isolated singularity, and singularities occurring on a very small interval of $ \psi$.

Furthermore, when $\theta \in \{ 0,\pi/4\}$ we are able to prove the existence of unbounded solutions.

\begin{theorem}
	\label{prop:GrowUp}
	Consider the non-trivial equilibrium $u_1$ for \eqref{eq:CGL}. Let $ \theta \in\{ 0 , \pi/4\}$. There exists a point in the unstable manifold of $u_1$ through which the solution is unbounded in positive time. 
\end{theorem}

This proof closely mimics a result obtained in \cite{stuke2017blow} and relies on the fact that the nonlinearity in  \eqref{eq:CGL} is holomorphic and \eqref{eq:CGL} generates an analytic semigroup when $ \theta  \in ( -\pi/2,\pi/2)$. 
We conjecture that the nonlinear Schr\"odinger equation ($\theta=\pi/2$) also has solutions that are unbounded.
Note also that this theorem does not distinguish between finite time blowup and infinite time growup.

The rest of the paper is organized as follows: In the following three sections, we briefly show three separate fundamental techniques of computer-assisted proofs. These have been introduced in previous studies \cite{takayasu2019rigorous,Jaquette2020}, and some minor modifications are made for the current problem. In Section~\ref{sec:UnstableManifold} we introduce the parameterization method for 1-complex dimensional unstable manifolds of equilibrium to \eqref{eq:CGL}. Section~\ref{sec:RigorousIntegrator} is devoted to introducing a method of rigorous integration of PDEs in forward time. Section~\ref{sec:TrappingRegion} is an introduction of constructing a trapping region near the zero equilibrium. Finally, we conclude in Section~\ref{sec:ResultsAndDiscussion}, where we present the proof of our main theorems (Theorems~\ref{prop:ManyHeteroclinics} and~\ref{prop:GrowUp}), and conclude with discussions and outlook of the present study.


\newcommand{\tu}{\tilde{u}}

\section{Parameterization of the Local Unstable Manifold} \label{sec:UnstableManifold}

As mentioned already in the Introduction, the proof of Theorem~\ref{prop:ManyHeteroclinics} relies on a rigorous parameterization of the unstable manifold of a steady state $u_1$ (see Figure \ref{fig:equilibria}) of the nonlinear equation \eqref{eq:CGL}. Note that $u_1$ satisfies $e^{i \theta } ( \Delta u_1 + u_1^2) = 0$. 
%
Solutions of \eqref{eq:CGL} with periodic boundary conditions may be expanded in Fourier series as
\begin{equation} \label{eq:Fourier_expansion}
	u(t,x) \bydef \sum_{k \in \Z} a_k(t) e^{2 \pi i k x},
\end{equation}
and matching like terms, this leads to the infinite system of nonlinear ODEs
\begin{equation} \label{eq:CGL_ODEs}
	\dot a_k(t) = g_k(a) \bydef \expig \left( -4 k^2 \pi^2 a_k+\left( a^2 \right)_k \right),
\end{equation}
where $a^2 = a*a$, which denotes the discrete convolution. As in \cite{Jaquette2020}, we assume that the Fourier coefficients in \eqref{eq:Fourier_expansion} satisfy the symmetry
\begin{equation} \label{eq:Fourier_symmetry}
	a_{-k}(t) = a_{k}(t), \quad \text{for all } k \in \Z.
\end{equation}
From the symmetry assumption, one only needs to solve for the non-negative Fourier coefficients. 
Let $g(a) = (g_k(a))_{k \ge 0}$ so that \eqref{eq:CGL_ODEs} can be more densely written as $\dot a = g(a)$. Denote by $\ta = (\ta_k)_{k \ge 0}$ the Fourier coefficients of the steady state $u_1(x)$ (already proven in \cite{Jaquette2020}), that is 
\[
u_1(x) = \ta_0 + 2 \sum_{k \ge 1} \ta_k e^{2 \pi i k x} = \sum_{k \in \Z} \ta_k e^{2 \pi i k x}.
\]
By construction, since $u_1$ is an equilibrium, its Fourier coefficients $\ta = (\ta_k)_{k \ge 0}$ satisfy $g(\ta)=0$.

Denote the unstable manifold of the steady states $\ta$ by
\[
W^u(\ta) \bydef \left\{ a^0 ~ \big| ~ \exists \text{ a solution } a(t) \text{ of  \eqref{eq:CGL_ODEs} with } a(0) = a^0 
\text{ and } \lim_{t \to -\infty} a(t) = \ta \right\}.
\]
Consider $(\tilde \lambda,\tilde h)$ the solution of the linearized problem \eqref{eq:linearized_eigprob} (as described in Remark~\ref{rem:fixing_the_eigenpair}) such that the real part of the eigenvalue $\tilde{\lambda}$ is positive. Under these assumptions, we introduce the theory to parameterize 1-complex dimensional unstable manifolds of steady states of the nonlinear equation \eqref{eq:CGL}. The approach is based on the Parameterization Method (e.g. see \cite{MR1976079,MR1976080,MR2177465}) for unstable manifolds, and follows closely the presentation of \cite{MR3906120}.

We are interested in solutions of \eqref{eq:CGL_ODEs} whose spectral representation have geometrically decaying Fourier coefficients. Given an geometric decay rate $\nu \ge 1$, this feature is encapsulated in the weighted $\ell^1$ Banach algebra
\begin{equation} \label{eq:ell_nu_one}
	\ell_\nu^1 \bydef \left\{ a = (a_k)_{k \ge 0} : |a|_{\nu} \bydef |a_0| + 2 \sum_{k=1}^\infty |a_k| \nu^k  < \infty \right\}.
\end{equation}
Indeed, given two sequences $a=(a_k)_{k \ge 0},b=(b_k)_{k \ge 0} \in \ell_\nu^1$, with their discrete convolution $a*b=((a*b)_k )_{k \ge 0}$ given component-wise by
\begin{equation} \label{eq:discrete_convolution}
	(a*b)_k = \sum_{k_1+k_2 = k \atop k_1,k_2 \in \Z} a_{|k_1|} a_{|k_2|},
\end{equation}
one has that $|a*b|_\nu \le |a|_\nu |b|_\nu$.

Denote by $\tb = (\tb_k)_{k \ge 0}$ the Fourier coefficients of the eigenvector $\tilde h(x)$ (which solves \eqref{eq:linearized_eigprob}), that is 
\[
\tilde h(x) = \tb_0 + 2 \sum_{k \ge 1} \tb_k e^{2 \pi i k x} = \sum_{k \in \Z} \tb_k e^{2 \pi i k x}.
\]
The Fourier eigenpair $(\tilde \lambda,\tb)$ satisfies $Dg(\ta) \tb = \lambda \tb$, which becomes, in the space of Fourier coefficients
\[
\expig \left( -4 k^2\pi^2 \tb_k + 2 \left(\ta*\tb \right)_k\right) -  \tilde \lambda \tb_k = 0, \qquad \tb_{-k} = \tb_{k}, \quad \text{for all } k \ge 0.
\]
Note that for $\ta$ and $\tb$, we assume that the symmetry condition \eqref{eq:Fourier_symmetry} holds.
The (computer-assisted) approach to obtain the steady state $\ta$ and the eigenpair $(\tilde \lambda,\tb)$ is already obtained and presented in \cite{MR3906120}. 

Denote by $\mathbb{D} = \{ z \in \C : |z| \le 1\}$ the closed unit disk centered at $0$ in the complex plane. 

Given an angle $\theta \in [-\pi/2,\pi/2]$, we look for a function $P_\theta:\mathbb{D}  \to \ell_\nu^1$ satisfying the {\em invariance equation}
\begin{equation} \label{eq:invariance_equation}
	g(P_\theta(\sigma)) = \tilde \lambda \sigma DP_\theta(\sigma),
\end{equation}
for all $\sigma \in \mathbb{D} $ subject to the first order constraints
\begin{equation} \label{eq:first_order_constraints}
	P_\theta(0)=\ta \quad \text{and} \quad DP_\theta(0) = \tb.
\end{equation}

If $P_\theta:\mathbb{D}  \to \ell_\nu^1$ solves \eqref{eq:invariance_equation} subject to the first order constraints \eqref{eq:first_order_constraints}, then one can easily verify that for every $\sigma \in \mathbb{D}$, the function
\begin{equation} \label{eq:unstable_solution}
	t \mapsto P_\theta(e^{\tilde \lambda t} \sigma)
\end{equation}
solves the differential equation on $\ell_\nu^1$ given by $\dot a = g(a)$ on $(-\infty,0)$. Moreover, since $a(0) = P_\theta(\sigma) \in image(P_\theta)$ and $\lim_{t \to -\infty} P_\theta(e^{\tilde \lambda t} \sigma) = P_\theta(0)=\ta$, then $P_\theta$ parameterizes a subset of the unstable manifold for $\ta$ given by
\begin{equation} \label{eq:local_unstable_set}
	W_{\rm loc}^{u,\tilde \lambda,\theta}(\ta,\tb) \bydef
	image(P_\theta)=P_\theta(\mathbb{D}) \subset W^u(\ta).
\end{equation}

We call the set $W_{\rm loc}^{u,\tilde \lambda,\theta}(\ta,\tb)$ in \eqref{eq:local_unstable_set} the {\em local unstable manifold at $\theta$ associated to $\tilde \lambda$ in the direction $\tb$}. Note that $W_{\rm loc}^{u,\tilde \lambda,\theta}(\ta,\tb)$ is a sub-manifold of the unstable manifold $W^u(\ta)$.

The strategy is now clear: compute a function $P_\theta:\mathbb{D}  \to \ell_\nu^1$ which solves \eqref{eq:invariance_equation} subject to the first order constraints given in \eqref{eq:first_order_constraints}. 
We represent $P_\theta$ as a power series
\begin{equation} \label{eq:P_power_series}
	P_\theta(\sigma) = \sum_{m \ge 0} p_m \sigma^m, \qquad p_m \in \ell_\nu^1.
\end{equation}
The coefficients of the power series \eqref{eq:P_power_series} are elements of the Banach space 
\begin{equation} \label{eq:X_nu}
	X^\nu \bydef  \left\{ p = \{p_m\}_{m \ge 0} : p_m \in \ell_\nu^1 \text{ and } 
	\| p\|_\nu \bydef \sum_{m \ge 0} |p_m|_\nu < \infty \right\}.
\end{equation}
%
%
Define the {\em Taylor-Fourier product} $*_{TF} :X^\nu \times X^\nu \to X^\nu$ as follows: given $p,q \in X^\nu$,
\begin{equation}
	(p *_{TF} q)_m \bydef \sum_{\ell =0}^m p_\ell * q_{m-\ell},
\end{equation} 
where $*:\ell_\nu^1 \times \ell_\nu^1 \to \ell_\nu^1$ is the standard discrete convolution \eqref{eq:discrete_convolution}. 
One can verify that $(X^\nu,*_{TF})$ is a Banach algebra, that is given $p,q \in X^\nu$,
\begin{equation} 
	\|p *_{TF} q \|_\nu \le \|p \|_\nu \|q \|_\nu.
\end{equation}

Plugging the power series \eqref{eq:P_power_series} in the invariance equation \eqref{eq:invariance_equation} and imposing the first order constraints \eqref{eq:first_order_constraints} leads to look for a solution $p \in X^\nu$ of $f(p)=0$, where the map $f = (f_{k,m})_{k,m \ge 0}$ is given by
\begin{equation} \label{eq:f=0_manifold}
	f_{k,m}(p) \bydef 
	\begin{cases}
		p_{k,0} - \ta_k, & m=0, \quad k \ge 0
		\\
		p_{k,1} - \tb_k, & m=1, \quad k \ge 0
		\\
		\left(\tilde \lambda m + 4 \expig k^2 \pi^2 \right) p_{k,m} - \expig (p*_{TF}p)_{k,m}, & m \ge 2, \quad k \ge 0.
	\end{cases}
\end{equation}
%

The following result shows that a zero of $f$ provides the Fourier-Taylor coefficients of the expansion of a parameterization of the local unstable manifold at $\theta$ associated to $\tilde \lambda$ in the direction $\tb$.

\begin{lemma} 
	Let $\nu > 1$ and suppose that $p = (p_{k,m})_{k,m \ge 0} \in X^\nu$ solves $f(p)=0$. For each $m\ge0$, denote $p_m = (p_{k,m})_{k \ge 0}$. Then, for each $\sigma \in \mathbb{D}$,
	\[
	P(\sigma) \bydef \sum_{m \ge 0} p_m \sigma^m \in W^u(\ta).
	\]
\end{lemma}

The previous result paves the way to come, that is to prove the existence of $p \in X^\nu$ such that $f(p)=0$. This is done using the following  Newton-Kantorovich type theorem, which we now introduce. Before that, denote by $B_r(y) \bydef \{ q \in X^\nu : \| q - y \|_{\nu} \le r\}$ the closed ball of radius $r>0$ centered at $y \in X^{\nu}$. 


\begin{theorem}[{\bf A Newton-Kantorovich type theorem}] \label{thm:radii_polynomials}
	Let $A \in B(X^{\nu},X^\nu)$ be a bounded linear operator (typically an approximate inverse of $Df(\bp)$).
	Consider a point $\bp \in X^\nu$ (typically a numerical approximation), assume that $A$ is injective and that $A f \colon X^\nu \to X^\nu$.
	Let $Y_0$ be a nonnegative constant, and a function $Z:(0,\infty) \to (0,\infty)$ satisfying
	\begin{align}
		\label{eq:general_Y_0}
		\| A f(\bp) \|_\nu &\le Y_0
		\\
		\label{eq:general_Z}
		\| I - A Df(q)\|_{B(X^\nu)} &\le Z(r), \quad \text{for all } q \in B_r(\bp),
	\end{align}
	where $\| \cdot \|_{B(X^\nu)}$ denotes the operator norm.  Define the radii polynomial by 
	\begin{equation} \label{eq:general_radii_polynomial}
		p(r) \bydef rZ(r) - r + Y_0.
	\end{equation}
	If there exists $r_0>0$ such that $p(r_0)<0$, then there exists a unique $\tp \in B_{r_0}(\bp)$ such that $f(\tp) = 0$.
\end{theorem}

Since the construction of the bounds $Y_0$ and $Z(r)$ satisfying \eqref{eq:general_Y_0} and \eqref{eq:general_Z}, respectively, is almost exactly the same as the ones presented in \cite{Jaquette2020}, we omit their construction here.

In practice, the rigorous computation of $W_{\rm loc}^{u,\tilde \lambda,\theta}(\ta,\tb)$, the local unstable manifold at $\theta$ associated to $\tilde \lambda$ in the direction $\tb$ is done as follows. Starting with the Fourier coefficients of the steady state $\ta$ and the eigenpair $(\tilde \lambda,\tb)$, we consider a finite dimensional reduction of $f = (f_{k,m})_{k,m \ge 0}$ defined in \eqref{eq:f=0_manifold} on which we apply Newton's method to obtain a numerical approximation $\bp = (\bp_{k,m})_{k=0,\dots,N \atop m = 0,\dots,M}$, yielding a Fourier-Taylor coefficients of the approximation

\begin{equation} \label{eq:P_power_series_approx}
	\bar P_\theta(\sigma) \bydef \sum_{m = 0}^M \bp_m \sigma^m, \qquad \bp_m = ((\bp_m)_k)_{k=0}^N \in \C^{N+1}.
\end{equation}

We wrote a MATLAB program (available at \cite{bib:codes}) which rigorously computes (i.e. controlling the floating point errors) the bounds $Y_0$ and $Z(r)$ as presented in Theorem~\ref{thm:radii_polynomials}.  The program uses the interval arithmetic library INTLAB available at \cite{Ru99a}.
Fixing $\nu =1$ and using the MATLAB program, the code verifies the existence of $r_0>0$ such that $p(r_0)<0$. By Theorem~\ref{thm:radii_polynomials}, there exists a unique $\tp \in B_{r_p}(\bp)$ such that $f(\tp) = 0$. By construction, the resulting function $\tilde P_\theta:\mathbb{D} \to \ell_\nu^1$ given by 
\[
\tilde P_\theta(\sigma) \bydef \sum_{m \ge 0} \tp_m \sigma^m
\]
parameterizes the local unstable manifold at $\theta$ associated to $\tilde \lambda$ in the direction $\tb$. Moreover, for each $\sigma \in \mathbb{D}$, $\tilde P_\theta(\sigma) \in W^u(\ta)$. We get the rigorous $C^0$ error bound
\[
\sup_{\sigma \in \mathbb{D}} | \tilde P_\theta(\sigma) - \bar P_\theta(\sigma)  |_\nu 
= \sup_{\sigma \in \mathbb{D}} \left| \sum_{m \ge 0} \tp_m \sigma^m - \sum_{m \ge 0} \bp_m \sigma^m  \right|_\nu \le \sup_{\sigma \in \mathbb{D}} \sum_{m \ge 0} |\tp_m - \bp_m|_\nu |\sigma|^m \le \| \tp - \bp \|_\nu \le r_0.
\]
We applied this approach to obtain the parameterizations of $W_{\rm loc}^{u,\tilde \lambda,\theta}(\ta,\tb)$ for the values  $\theta \in \{0,\pi/4,\pi/2\}$. As pointed out already in \cite{Jaquette2020} (see Remark 3.3), the continuity of the bounds $Y_0$ and $Z(r)$ in the parameter $\nu \ge 1$ implies that if the computer-assisted proof is performed for $\nu=1$, then, there exists $\tilde \nu>1$ such that $p(r_0)<0$, and hence, the unique solution of $f=0$ satisfies $\tp \in X^{\tilde \nu}$. This in turn implies that for each $\sigma \in \mathbb{D}$, $\tilde P_\theta(\sigma) \in \ell_{\tilde \nu}^1$ and therefore the corresponding points on $W_{\rm loc}^{u,\tilde \lambda,\theta}(\ta,\tb)$ are analytic in space.

\section{Rigorous Integrator} \label{sec:RigorousIntegrator}

In this section we briefly introduce a method of rigorous numerical integration for \eqref{eq:CGL} in forward time. This is firstly proposed in \cite{takayasu2019rigorous} and then slightly improved for integrating the solution of the nonlinear Schr\"odinger (NLS) case ($\theta=\pi/2$) in \cite{Jaquette2020}. For proving the existence of heteroclinic orbits for \eqref{eq:CGL}, this integrator will propagate the rigorous inclusion of 1-complex dimensional unstable manifold $W_{\rm loc}^{u,\tilde \lambda,\theta}(\ta,\tb)$, which is given by the parameterization method described in Section \ref{sec:UnstableManifold}.

Let us assume that the initial data is presented by the Fourier series
\[
u(x,0)=\sum_{k \in \Z}\varphi_k e^{ik\omega x},\quad \varphi=(\varphi)_{k\in\mathbb{Z}}\in \ell^1_\nu,
\]
where $\ell^1_\nu$ is the same sequence space as in the previous section, but the definition of its norm is slightly changed as $|a|_{\nu} \bydef\sum_{k \in \Z}|a_k|\nu^{|k|}$ because there is no need to assume the symmetry \eqref{eq:Fourier_symmetry} in the followings.
Under the Fourier ansatz of $u$ defined in \eqref{eq:Fourier_expansion}, we consider the Cauchy problem of \eqref{eq:CGL_ODEs} with the initial data $a(0)=\varphi$.
For a fixed time $h>0$, let $J\bydef[0,h]$ be a time step.
The Banach space in which we validate rigorously the local existence of the solution is defined by
\[
X\bydef C(J;\ell^1_\nu),\quad \|a\|_X\bydef\sup_{t\in J}|a(t)|_\nu,\quad|a(t)|_\nu=\sum_{k \in \Z}|a_k(t)|\nu^{|k|}~\mbox{for a fixed}~t\in J.
\]
Let
$\bar{a}(t) \bydef\left(\ldots, 0,0, \bar{a}_{-K}(t), \ldots, \bar{a}_{K}(t), 0,0, \ldots\right)$
be a truncated approximation of $a(t)$ with the Fourier projection $K>0$.
We rigorously include the solution $a(t)$ of \eqref{eq:CGL_ODEs} in a neighborhood of $\bar{a}(t)$ defined by
\[
B_{J}(\bar{a}, \varrho) \bydef\left\{a \in X:\|a-\bar{a}\|_{X} \leq \varrho,~a(0)=\varphi\right\}.
\]

To represent the ODEs \eqref{eq:CGL_ODEs} more densely, we define the \emph{Laplacian} operator $L$ acting on $b=(b_k)_{k\in\mathbb{Z}}$ as
\[
Lb\bydef\left(-4 \pi^2 k^2b_k\right)_{k\in\mathbb{Z}},
\]
and $D(L)\subset\ell^1_\nu$ denotes the domain of the operator $L$.
For $a\in C^1(J;D(L))$ let us define
\begin{align}\label{eq:def_F}
	(F(a))(t)\bydef \dot{a}(t) - g(a(t)) = \dot{a}(t)-\expig\left(La(t)+a^2\left(t\right)\right).
\end{align}
Then we consider the ODEs \eqref{eq:CGL_ODEs} as the zero-finding problem $F(a)=0$.
That is, if $F(a)=0$ with the initial condition $a(0)=\varphi$ holds, such $a(t)$ solves the Cauchy problem of \eqref{eq:CGL_ODEs}.
We also define an operator $T:X\to X$ as
\begin{align}\label{eq:simp_Newton_op}
	(T(a))(t)\bydef U(t,0)\varphi+\expig\int_0^tU(t,s)\left(a^2(s)-2\bar{a}(s)*a(s)\right)ds,
\end{align}
where $\left\{U(t,s)\right\}_{0\le s\le t\le h}$ is the evolution operator (cf., e.g.,\cite{pazy1983semigroups}) on $\ell^1_\nu$.
Such evolution operator is defined by a solution map of the linearized problem of \eqref{eq:CGL_ODEs} at $\ba$,
\begin{equation}\label{eq:linearized_problem}
	\dot{b}_k(t)+ 4  \pi^2 \expig k^2 b_k(t)-2\expig \left(\bar{a}\left(t\right)*b(t)\right)_k=0\quad(k\in\mathbb{Z})
\end{equation}
with any initial data $b(s)=\psi\in\ell^1_\nu$ ($0\le s\le t$). In other words, the evolution operator provides the solution of \eqref{eq:linearized_problem} via the relation $b(t)=U(t,s)\psi$.
In \cite[Theorem 3.2]{takayasu2019rigorous}, we give a hypothesis to validate existence of the evolution operator.
Checking the hypothesis by using rigorous numerics, we also provide a uniform bound of the evolution operator over the simplex $\cS_h\bydef \{(t,s):0\le s\le t\le h\}$, that is a computable constant $\bm{W_h}>0$ satisfying
\begin{equation}\label{eq:W_h_constant}
	\|b\|_{X}=\sup_{(t,s)\in\cS_h}|U(t,s)\psi|_\nu\le \bm{W_h}|\psi|_\nu,\quad\forall\psi\in\ell^1_\nu.
\end{equation}
Such a constant is obtained by solving the linearized problem \eqref{eq:linearized_problem} via decomposing the solution $b$ by the finite mode $b^{(K')}=\left(b_{k}\right)_{|k|\le K'}$ and the tail  $b^{(\infty)}=\left(b_{k}\right)_{|k|>K'}$ for $K'\in\mathbb{N}$ satisfying $K>K'$.

The rigorous integrator we used for proving heteroclinic orbits is based on the following theorem, which provides the local existence (in time) of the solution via interval techniques.
\begin{theorem}[\!\!{\cite[Theorem 4.1]{takayasu2019rigorous}}]\label{thm:local_inclusion}
	Given the approximate solution $\ba\in C(J;D(L))\cap C^1(J;\ell^1_\nu)$ of \eqref{eq:CGL_ODEs} and the initial sequence $\varphi$, assume that $|\varphi-\ba(0)|_\nu\le\varepsilon$ holds for $\varepsilon\ge 0$.
	Assume also that, for any $a\in B_J\left(\ba,\varrho\right)$,
	$\left\| T(a)-\ba\right\|_X\le f_{\varepsilon}\left(\varrho\right)$ holds,
	where $f_{\varepsilon}(\varrho)$ is given by
		\begin{align}\label{eq:f}
	f_{\varepsilon}\left(\varrho\right)\bydef \bm{W_h}\left[\varepsilon+h\left(2\varrho^2+\delta\right)\right].
		\end{align}
	Here, $\bm{W_h}>0$ and $\delta> 0$ satisfy
	$\sup_{(t,s)\in\mathcal{S}_h}\left\|U(t,s)\right\|_{B (\ell^1_\nu)}\le \bm{W_h}$ and
	$\left\|F(\ba)\right\|_{X}\le\delta$, respectively.
	If there exists $\varrho_{0}>0$ such that
		\[
	f_{\varepsilon}\left(\varrho_0\right)\le\varrho_0,
		\]
	then the exact Fourier coefficients $\ta$ of the solution of \eqref{eq:CGL_ODEs} are rigorously included in $B_J\left(\ba,\varrho_0\right)$ and are unique in $B_J\left(\ba,\varrho_0\right)$.
\end{theorem}
The proof of Theorem \ref{thm:local_inclusion} is based on Banach's fixed-point theorem for the operator $T$ defined in \eqref{eq:simp_Newton_op}. For more details, we refer to our previous papers \cite[Section 4]{takayasu2019rigorous} for the case of $\theta=0, \pi/4$ and \cite[Section 4]{Jaquette2020} for $\theta =\pi/2$.

We applied this theorem iteratively to extend the local inclusion of the solution over longer time intervals. 
Let $0=t_0<t_1<\dots$ be grid points in time.
We call $J_i \bydef [t_{i-1}, t_i]$ the $i^{th}$ time step. Let us also define $t_i \bydef ih_i$ ($i = 1, 2,\dots$) with the stepsize $h_i$ of $J_i$, which can be changed adaptively.
Firstly, we assume that the solution $a(t)$ of \eqref{eq:CGL_ODEs} is rigorously included in $B_{J_1}\left(\ba^{J_1}, \varrho_{1}\right)$, where $\ba^{J_1}$ denotes an approximate solution given in $J_1$.
Secondly, we set the next time step $J_2$ in order to obtain the uniform bound of evolution operator $\bm{W_{h_2}}>0$.
Thirdly, the initial data is updated by the endpoint of solution on $J_1$, that is $\varphi = a(t_1)$. 
Replacing $J=J_2$, we apply Theorem \ref{thm:local_inclusion} for the Cauchy problem on $J_2$. 
We validate the sufficient condition of Theorem~\ref{thm:local_inclusion} and then obtain the next inclusion $B_{J_2}\left(\ba^{J_2}, \varrho_{2}\right)$. Finally, go back to the second step and repeat these processes several times. Thus, we can extend the local inclusion of solution until the sufficient condition of Theorem \ref{thm:local_inclusion} is no longer satisfied.
 This \emph{time stepping} scheme works successfully for the complex-valued nonlinear heat equations with the cases $\theta \in\{0, \pi / 4, \pi / 2\}$.

To complete the proof of existence of heteroclinic orbits, we need to prove the global existence of the solution after the numerical verification of the local inclusion using the time stepping. We construct a trapping region near the zero equilibrium, which guarantees the global existence of solution converging to zero.
The next section is devoted to introducing how we construct such a trapping region rigorously.

\section{Trapping Region} \label{sec:TrappingRegion}

If the zero equilibrium to \eqref{eq:CGL} was linearly stable then it would be relatively  straightforward to construct a trapping region, that is an open set within which every initial condition will converge to the zero equilibrium in forward time. 
However this is not the case. 
Indeed, from \eqref{eq:CGL_ODEs} one sees that the linearized flow has eigenvalues $ - 4 \pi^2 e^{i \theta} k^2$ for all $ k \in \Z$.  
For all $ \theta$ there exists a single zero eigenvalue, which presents a significant obstacle in identifying a trapping region.

The eigenspace associated with the zero eigenvalue corresponds to spatially constant solutions, and naturally forms an invariant subspace of the nonlinear dynamics. 
The dynamics on this subspace are governed by the  ordinary differential equation $\dot{z} = e^{i \theta} z ^2$, whose solutions with initial condition $\zeta(0) = z_0$ may be explicitly given by $ \zeta: \R_+ \to \C$    below 
\begin{align} \label{eq:PpowerSolution}
	\zeta(t) &= \frac{  z_0 }{1 -   z_0 t e^{i \theta} }.
\end{align}
Geometrically, these solutions foliate the origin by homoclinic orbits, with the exception of initial conditions  $ \pm r e^{i \theta}$ which blowup in finite positive/negative time.

While the center manifold associated with the zero eigenvalue contains solutions which blowup in finite time,  an open set of initial conditions solutions will converge to zero with monotonically decreasing norm. 
To that end we make the following definition.  
 \begin{definition}
 	For any $ \rho_0 > 0$ define the region 
 	\begin{align}\label{eq:InvariantBall}
 		B(\rho_0) \bydef 
 		\{
 		z \in \C :  
 		Re(e^{i \theta } z) \leq 0;
 		| z| \leq \rho_0
 		\} .
 	\end{align}
 \end{definition}

To then develop a trapping region about the zero equilibrium, we will restrict to initial data whose zeroth Fourier coefficients are contained in some region $B(\rho_0)$. 
When  $ \theta \in (-\pi/2,\pi/2)$ then all of the other eigenvalues in the linearization have negative real part. 
In \cite{takayasu2019rigorous}, we characterized an explicit trapping region about the zero equilibrium using a Lyapunov-Perron argument to demonstrate a stable foliation of  the region $B(\rho_0)$. However this analysis depended on the eigenvalues having a negative real part, and could not be extended to the case $ \theta = \pm \pi/2$. 

Treating the NLS case $ \theta = \pm \pi/2$ (where all eigenvalues lie on the imaginary axis) requires a different approach. 
In \cite{Jaquette2020} a theorem was proved which identifies a trapping region about the zero equilibrium, by blowing up the dynamics about the spatially homogeneous solutions.  
This approach does not take advantage of any exponential decay in the higher Fourier modes, as it was constructed to treat the NLS case where no such decay is present.  
However in the case where  $ \theta \in (-\pi/2,\pi/2)$ the trapping region argument in \cite{Jaquette2020} can be readily adapted, as the exponential decay in the higher Fourier modes only assists in the analysis. 

Let $ \iota^0: \C \hookrightarrow  \ell_\nu^1$ denote the inclusion into the 0\textsuperscript{th} Fourier mode.
Below we present the trapping region theorem from \cite{Jaquette2020} adapted for the case when one is integrating in the complex plane of time. The proof is trivially different, and is left to the reader.

\begin{theorem}[cf \cite{Jaquette2020} Theorem 2.3]\label{thm:HomoclinicBlowup} 
	Consider \eqref{eq:CGL_ODEs} with  $ \theta \in [-\pi/2,\pi/2]$. Fix $ 0 <  \rho_0, \rho_1$ and define the set 
	\begin{align} \label{eq:BallOfFunction}
		\cB(\rho_0,\rho_1) &\bydef 
		\left\{
		\phi + 	\iota^0(z_0)  \in \ell_{\nu}^1 \; 
		:\;
		z_0 \in B(\rho_0)  \subseteq \C;~
		| \phi |_\nu  \leq \rho_1 |z_0|^2
		\right\}.
	\end{align}
	If there exists some $ r >0$ such that 
	\begin{align} \label{eq:RadiiExponential}
		\rho_1  \exp \left\{  \tfrac{\pi}{2} 		r \rho_0 
		 \right\} < r,
	\end{align}
	then solutions of points  $ a(0) = 	  \phi + 	\iota^0(z_0)\in \cB(\rho_0,\rho_1) $ under \eqref{eq:CGL_ODEs} will exist for all positive time,   converge to zero, and satisfy $ | a(t) - \iota^0(\zeta(t)) |_\nu \leq r | \zeta(t)|^2$ for $\zeta(t)$ in \eqref{eq:PpowerSolution}.
\end{theorem}

In short, if parameters $r,\rho_0,\rho_1$ are such that the inequality \eqref{eq:RadiiExponential} is satisfied, then $\cB(\rho_0,\rho_1)$ is a trapping region. 
As the left-hand side of \eqref{eq:RadiiExponential} is monotonically increasing in $\rho_0$ and $\rho_1$, it is ideal to take these parameters as small as possible. 
For an explicit sequence $\bar{a} \in \ell_\nu^1$, one may readily compute values of  $ \rho_0$, $\rho_1$ for which $ \bar{a} \in \cB(\rho_0,\rho_1)$, and then search for a value of $r$ for which \eqref{eq:RadiiExponential} is satisfied. 
A similar procedure may be done when considering a ball about $ \bar{a}$, and we refer to \cite{Jaquette2020} for further details.


\section{Results and discussions}
\label{sec:ResultsAndDiscussion}

Finally, combining the above three techniques, namely the parameterization of the local unstable manifold (Section~\ref{sec:UnstableManifold}), the rigorous integrator (Section~\ref{sec:RigorousIntegrator}) and the computation of the trapping region (Section~\ref{sec:TrappingRegion}), we present the proofs of our main theorems. The proof of Theorem \ref{prop:ManyHeteroclinics} is completed with computer assistance. 
All the codes for generating the following results with computer-assisted proofs are available at \cite{bib:codes}.

\subsection{Proof of Theorem \ref{prop:ManyHeteroclinics} and \ref{prop:GrowUp}}


\begin{proof}[Proof of Theorem \ref{prop:ManyHeteroclinics}]
	Our proof of existence of heteroclinic orbits consist of three parts. Firstly we construct the parameterization of the strong unstable manifold associated with the eigenpair $(  e^{i \theta} \tilde{\lambda}, e^{i \psi_k} \tilde{h})$ introduced in Section \ref{sec:UnstableManifold}.
	Secondly rigorous integrator provided in Section \ref{sec:RigorousIntegrator} starts from the rigorous inclusion of the endpoint of unstable manifold. The integrator propagates the inclusion of a point on the unstable manifold in forward time using the time stepping scheme. Finally, we validate that the hypothesis of Theorem \ref{thm:HomoclinicBlowup} holds after several time stepping, that is the solution of the Cauchy problem enters the trapping region described in Section~\ref{sec:TrappingRegion}. Then the whole orbit connects the equilibrium $u_1$ to zero. This completes the proof.

	(a) For $\theta= 0$, our computer-assisted approach succeeded in proving the heteroclinic orbits with indices $k=0,\dots,320, 358,359$. The results of validation are shown in Figure \ref{fig:result_CGL_0} (right). We have almost the same solution distribution as the one given by non-rigorous numerics. For indices $321\le k\le 357$, Our computer-assisted approach failed to prove the global existence of solution (in the third step). There are two reasons for such failure. One is that the propagation of errors causes the validation to fail. This is a limitation of the rigorous integrator. Further improvements of the integrator are needed to better control the propagation of errors. The other is that the asymptotic to an unbounded solution makes it impossible to validate the global existence of solutions. As presented in Figure \ref{fig:onion}, the solution profile appears to  blowup for some angles. Since our proof is based on rigorous numerics, we have strong evidence of existence of blow-up solutions.
	
	\begin{figure}[htbp]
		\begin{minipage}{0.5\hsize}
			\centering
			\includegraphics[width=\textwidth]{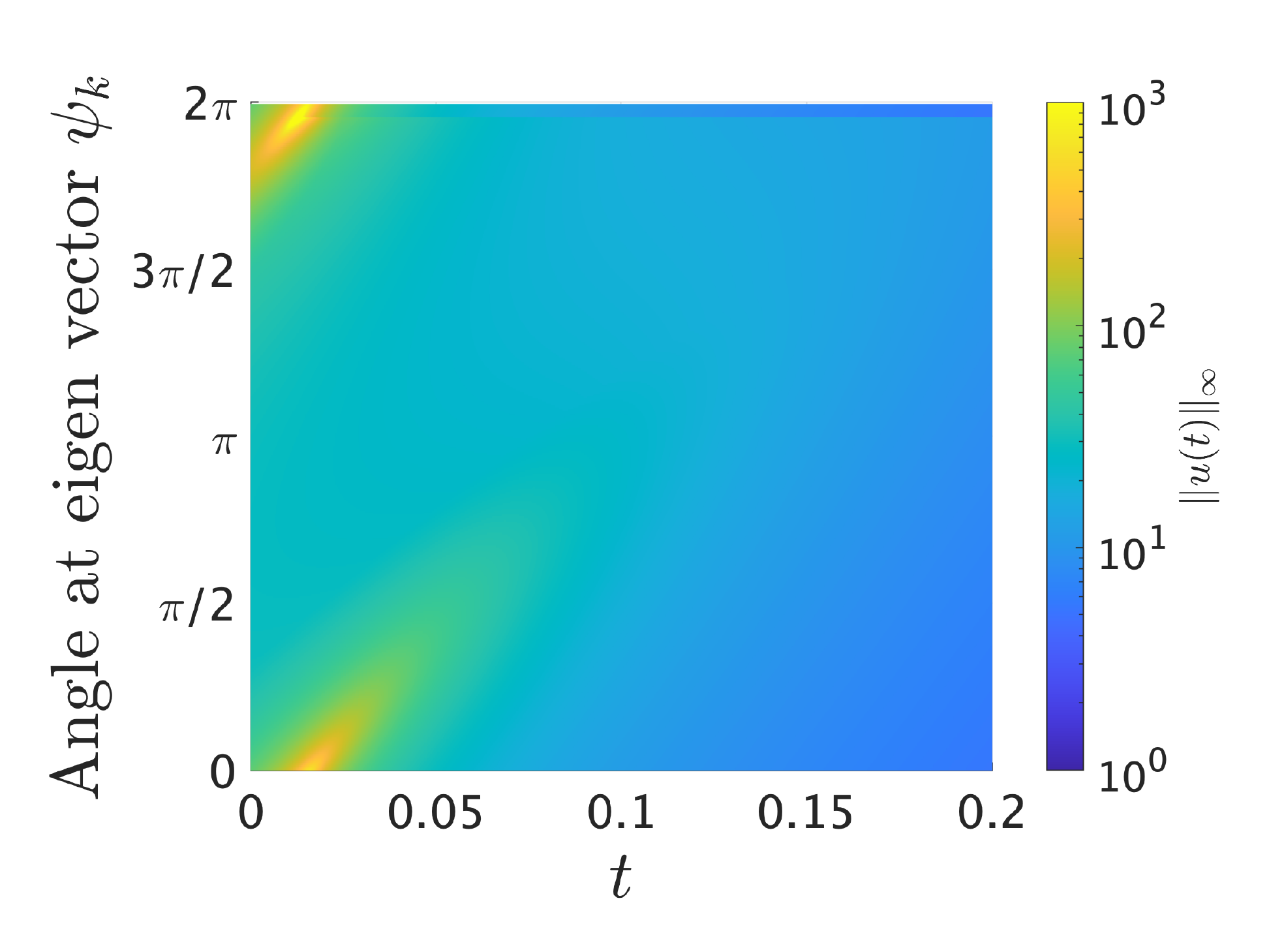}
		\end{minipage}
		\begin{minipage}{0.5\hsize}
			\centering
			\includegraphics[width=\textwidth]{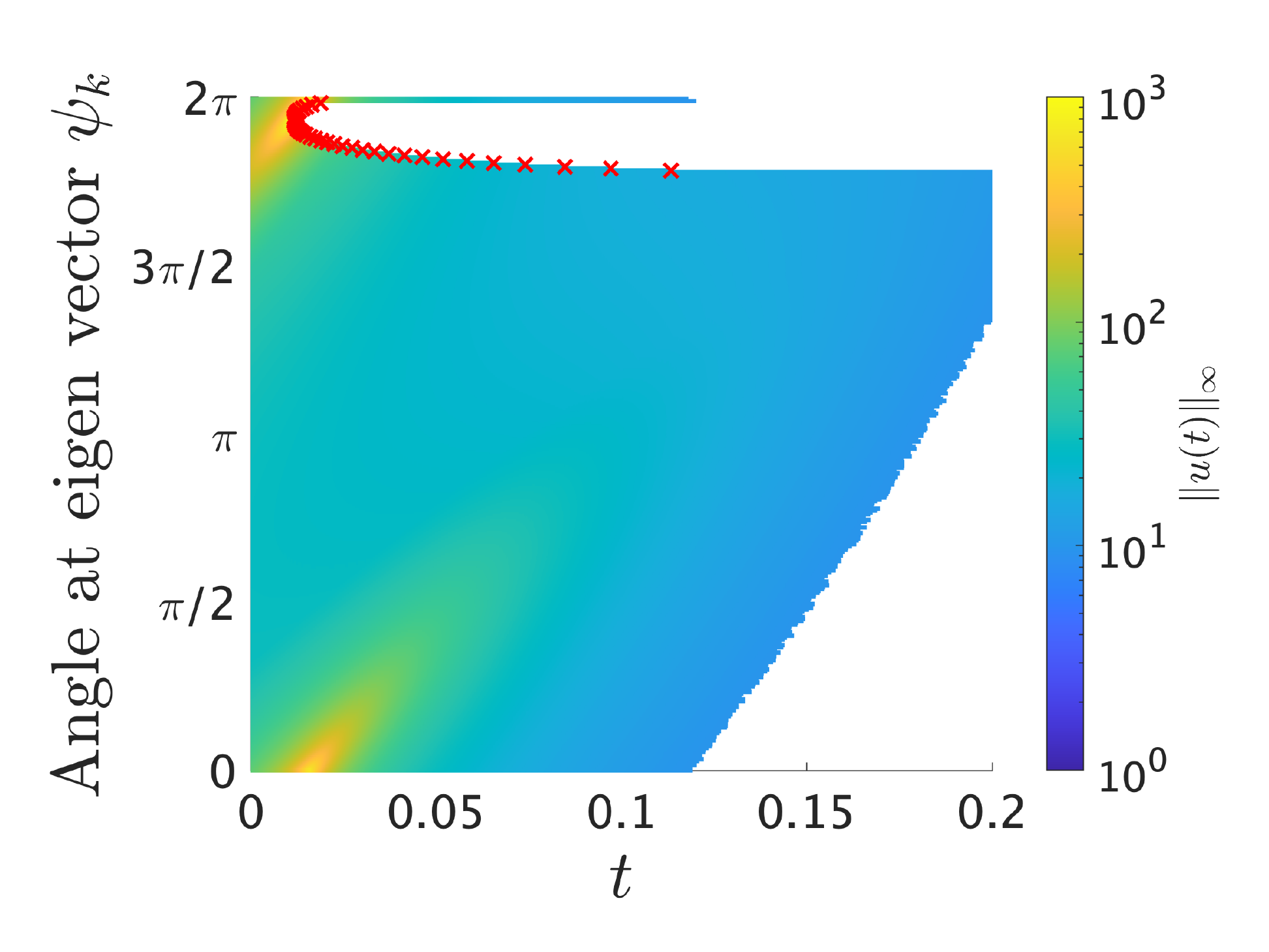}
		\end{minipage}
		\caption{Plotting the time-dependent supremum norm of the solution of \eqref{eq:CGL} in the case of $\theta=0$ (complex-valued nonlinear heat equations). (Left) Non-rigorously computed result with respect to the angle $\psi_k$ at the unstable eigenvector.  (Right) Results of rigorous integration with validation of global existence. If there is a red cross mark at the end of time (for the case of $\psi_k$ with $321 \leq k \leq 357$), the integrator fails in validating the solution due to error propagation. Unless cross-marked, it is successfully proved that the solution exist globally in time and that such a solution converges to the zero function.}
		\label{fig:result_CGL_0}
	\end{figure}
	
	(b) We show the results for the case of $\theta=\pi/4$ in Figure \ref{fig:result_CGL_pi/4} (right). As in the case of $\theta = 0$, the distribution of solutions is almost the same as the one obtained by non-rigorous numerics. Our computer-assisted approach succeeded in proving existence of heteroclininc orbits except for the indices $311 \leq k \leq 332$, which is slightly narrower compared with the case of $\theta=0$. This is because the amplitude of the solution profiles is smaller than that for $\theta=0$. The main difference from the case of $\theta=0$ is that the Morse index of the equilibrium is different (2 for $\theta=0$ and 1 for $\theta=\pi/4$). Other than that, the results of computer-assisted proofs were almost the same as the $\theta = 0$ case. We also have an evidence to conjecture that there could be blow-up solutions for some angles.
	
	\begin{figure}[htbp]
		\begin{minipage}{0.5\hsize}
			\centering
			\includegraphics[width=\textwidth]{figs/CGL_pi_4_pt1_solution_distribution}
		\end{minipage}
		\begin{minipage}{0.5\hsize}
			\centering
			\includegraphics[width=\textwidth]{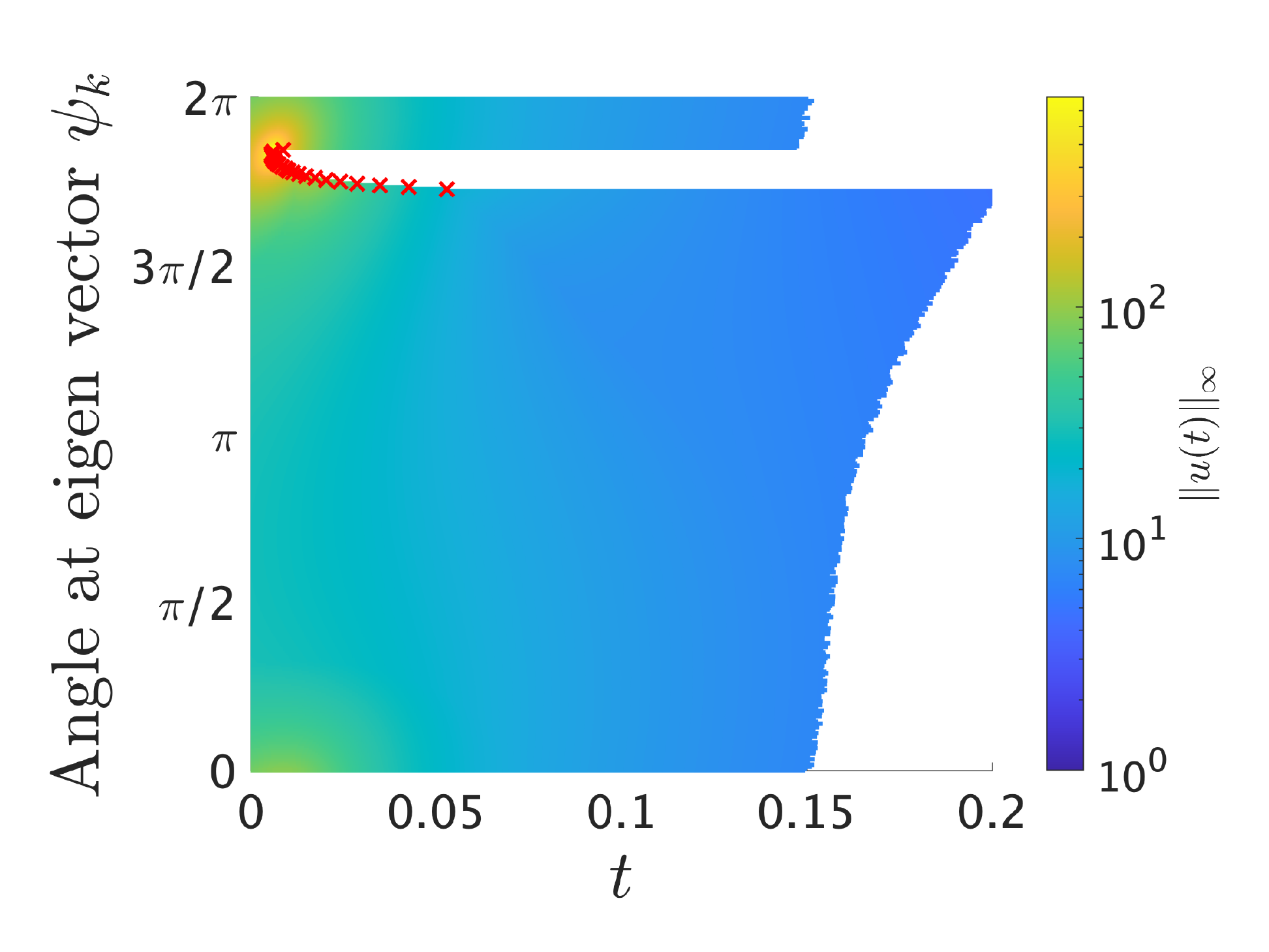}
		\end{minipage}
		\caption{Plotting the time-dependent supremum norm of the solution of \eqref{eq:CGL} in the case of $\theta=\pi/4$. (left) Non-rigorously computed result. (right) Results of rigorous integration with validation of global existence. The integrator fails for the case of $\psi_k$ with $311 \leq k \leq 332$. For other cases, the global existence of solutions converging to zero was successfully proved.}
		\label{fig:result_CGL_pi/4}
	\end{figure}
	
	(c) Finally, for $\theta=\pi/2$ (the nonlinear Schr\"odinger case), Figure \ref{fig:result_NLS} (b) and (c) show results of our computer-assisted proofs. 
	In this case the nontrivial equilibrium has one unstable eigenvalue, one stable eigenvalue, and infinitely many imaginary eigenvalues. Therefore it is easy to predict that the dynamics of the solution is different from the cases $\theta \in \{ 0,\pi/4\}$.
	Compared to the other two cases, the behavior of the solution is more complicated, which is shown in Figure \ref{fig:result_NLS} (a). The former solution appears to be transitioning smoothly, while the latter solution appears to be more oscillatory. Then such oscillatory behaving solution is difficult to validate, and failed in the range $214 \leq k \leq 321$. More precisely, the error propagation is so significant that it is difficult to integrate the solution rigorously (see Figure \ref{fig:result_NLS} (c)). However, our results, as in the other two cases, suggest that there could be a blowup of solutions at certain angles. \qedhere

	\begin{figure}[htbp]
		\centering
		\begin{minipage}{0.48\hsize}
			\centering
			\includegraphics[width=.95\textwidth]{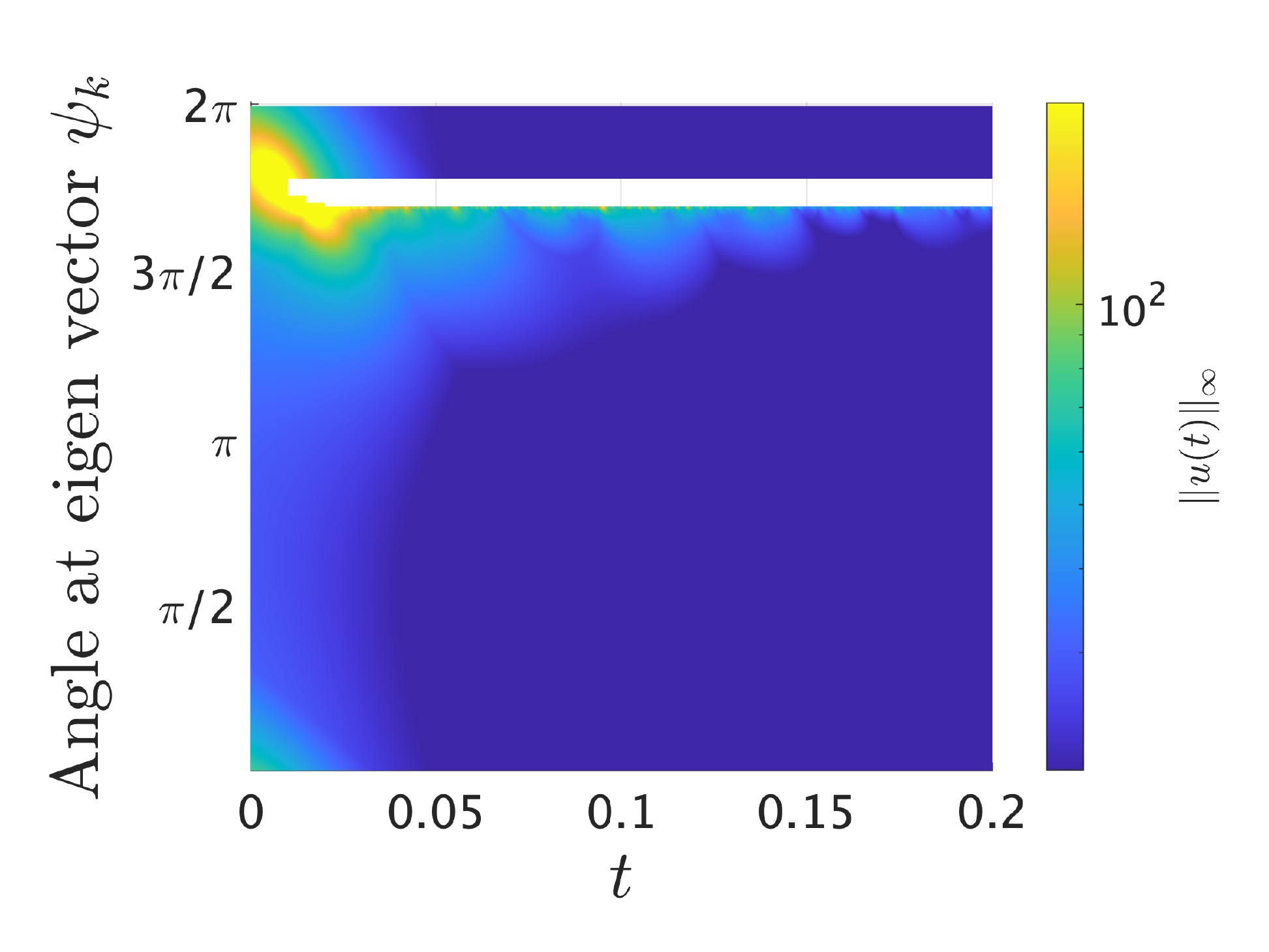}
			(a)
		\end{minipage}
		\begin{minipage}{0.48\hsize}
			\centering
			\includegraphics[width=\textwidth]{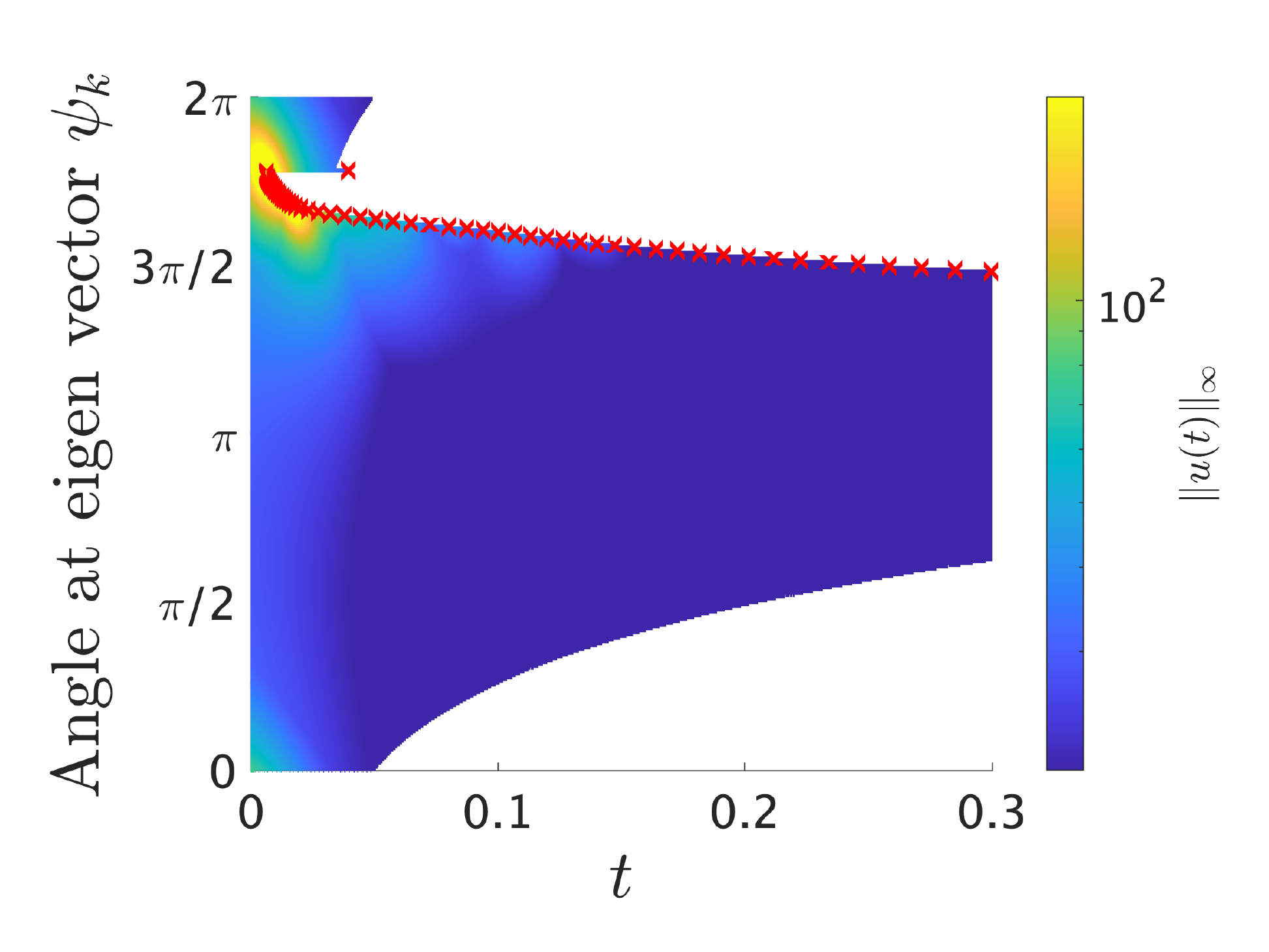}
			(b)
		\end{minipage}
		\begin{minipage}{0.5\hsize}
			\centering
			\includegraphics[width=\textwidth]{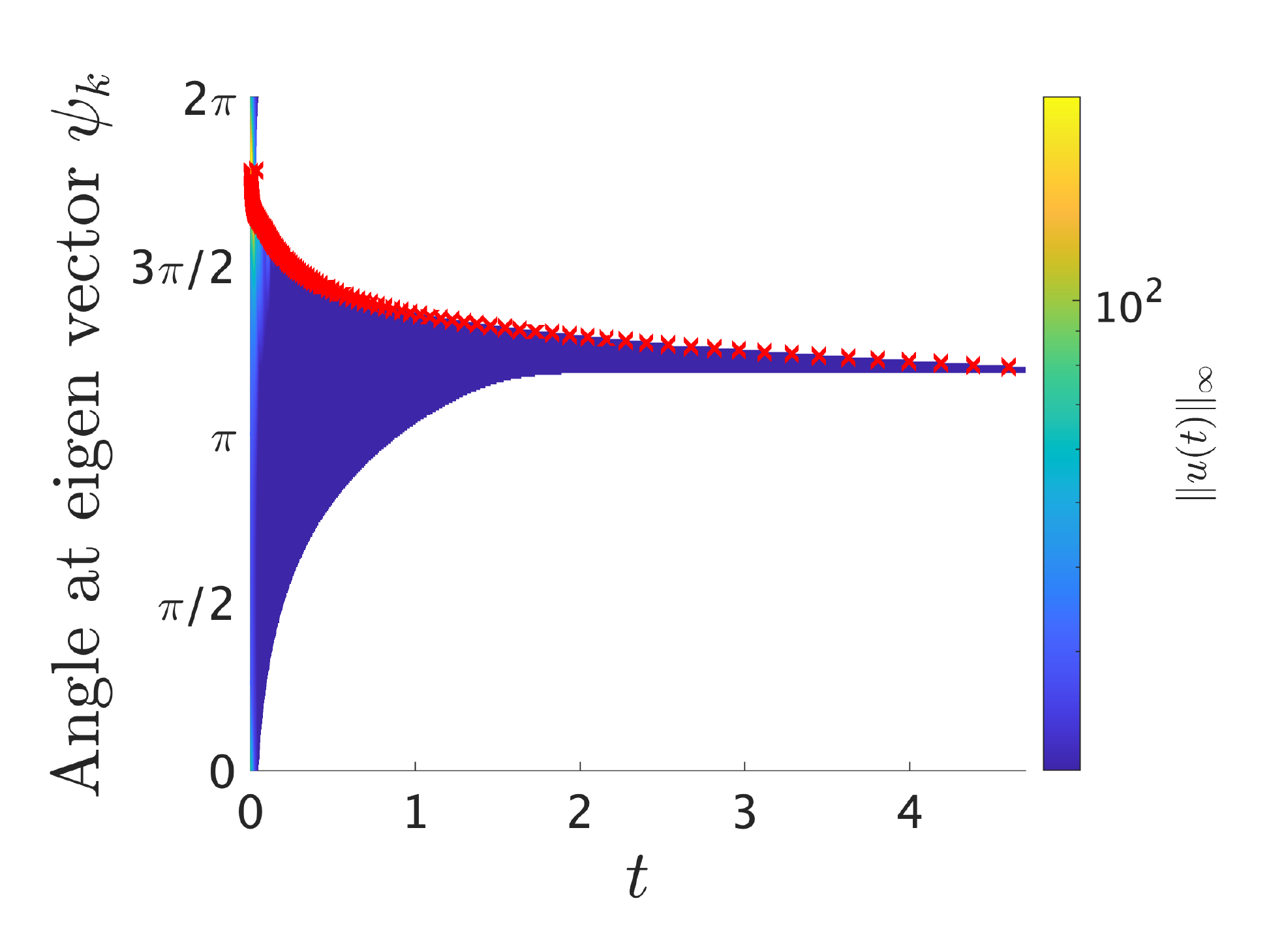}
			(c)
		\end{minipage}
		\caption{Plotting the time-dependent supremum norm of the solution of \eqref{eq:CGL} in the case of $\theta=\pi/2$ (the nonlinear Schr\"odinger equation). (a) Non-rigorously computed result. The dynamics appears to be more complicated than for the cases of $\theta\in (-\pi/2,\pi/2)$. (b) Results of rigorous integration with validation of global existence. The integrator fails for the case of $\psi_k$ with $214 \leq k \leq 321$. Due to the complicated dynamics, it has a wider range of integration failures than the other two cases. (c) Zooming out of the Figure (b). The error propagation prevents from rigorously integrating the solution. But, for some cases, the computer-assisted proof of global existence of solutions converging to zero was successfully obtained.}
		\label{fig:result_NLS}
	\end{figure}
	
\end{proof}

\begin{proof}[Proof of Theorem \ref{prop:GrowUp}]
	The proof follows from virtually the same argument as given in  \cite[Lemma 2.6]{stuke2017blow}, which we summarize below.   
	
	For $\nu \geq 0$ and the unit disk $ \mathbb{D}\subseteq \C$, let $P_\theta: \mathbb{D}  \to \ell^1_\nu$ denote a parameterization of the strong unstable manifold about the equilibrium $ u_1$ given in Theorem \ref{prop:ManyHeteroclinics}. 
	By the construction, as described in Section~\ref{sec:UnstableManifold}, the map $P_\theta$ is analytic. 
	Fix a continuous linear functional $\pi: \ell^1_\nu \to \C$ which separates $u_1$ and $0$, whereby $ \pi(u_1) \neq \pi(0)$ and $ \pi \circ P_\theta: \mathbb{D}\to \C$ is analytic. 
	
	Consider the flow map $ \Phi: \R_+ \times \ell^1_\nu \to \ell^1_\nu$ of the evolutionary equation \eqref{eq:CGL}. 
	If $|\theta|<\pi/2$ then $\Phi$ is an analytic semigroup.   
	For $ n \geq 0$ we define a sequence of projected time-$n$ maps 	$g_n: \mathbb{D} \to \C$ given by 
	\[
	g_n(z) \bydef \pi \circ \Phi(n, P_\theta(z) ).
	\]
	As this is a composition of analytic functions, each $ g_n$ is analytic.

	The proof of unbounded solutions proceeds by contradiction.  Suppose there is a uniform bound $K >0 $ such that $\sup_{z \in \mathbb{D}} |g_n(z)| \leq K$ for all $ n \geq 0$. 	 
	By the construction of our parameterization $P_\theta$, we have $ g_n(0) = \pi( u_1) =p^*$ for all $n \in \N$. 
	As we have demonstrated with computer-assisted proofs at parameters $\theta = 0, \pi/4$, there exists a heteroclinic point $q_0 \in \mathbb{D}$ such that $\lim_{t \to - \infty} \Phi(t,q_0) = u_1$ and $\lim_{t \to + \infty} \Phi(t,q_0) = u_0 \equiv 0$, the zero equilibrium.
	As the trapping region constructed in Theorem \ref{thm:HomoclinicBlowup} is an open set, then  there exists a neighborhood $U \subseteq \mathbb{D}$ about $ q_0$ such that all points in $U$ are heteroclinic points between $u_1$ and $u_0$.
	Hence $\lim_{ n \to \infty} g_n(q) =\pi(0) =0$ for all $q \in U$. 
	Recall that as $\pi$ separates $u_1$ and $u_0 \equiv 0$, then $ p^* \neq 0$. 
	
	Since the functions $g_n$ are uniformly bounded, by the Vitali theorem there exists an analytic function $ g: \mathbb{D} \to \C$ to which the sequence $ \{ g_n\}_{n=0}^\infty$ converges uniformly on compact subsets. 
	As argued above, $g(q) = \lim_{ n \to \infty} g_n(q) = 0$ for all $ q \in U$.  
	As $g$ is analytic and locally constant on the open set $U$, then  $g$ must be constant and equal $0$ on all of $ \mathbb{D}$. But this contradicts $g(0) =p^*$. 
	Thus the functions $g_n$ cannot be uniformly bounded for all $ n \geq 0$. 
	As $\pi$ was chosen to be an arbitrary continuous linear functional which separates $u_1$ and $0$, it follows that the globalization of the strong unstable manifold of $u_1$ must be unbounded.   
	
\end{proof}

Theorem \ref{prop:GrowUp} establishes the existence of solutions which limit to the non-trivial equilibrium in negative time, and are unbounded in positive time. 
However it does not distinguish whether such a solution exists for all time, a so-called growup solution, or blows up in finite time. 
Note also that Theorem \ref{prop:GrowUp} is stated just for $ \theta =0, \pi / 4$,  because we cite Theorem \ref{prop:ManyHeteroclinics} for the existence of  a heteroclinic orbit. 
For other values of $ \theta \in (-\pi / 2 , \pi /2)$ the existence of an heteroclinic orbit from $u_1$ to $0$  similarly would force the existence of unbounded solutions. 
However this argument does not readily extend to the case $\theta = \pm \pi/2$. 
This is because  when $ \theta = \pm \pi/2$   then \eqref{eq:CGL} becomes a nonlinear Schr\"odinger equation, and the flow map $\Phi$ generates a continuous semigroup, but not an analytic semigroup.


\subsection{Discussion and Outlook}

Numerics are often indispensable in understanding nonlinear evolutionary equations.  
However the reliability of these techniques are significantly affected when trajectories are integrated for a long time or their norm becomes very large.  
This is necessarily the case when studying unstable manifolds or blowup solutions, so it is important to know how much we can trust numerics in these situations.

In this work, we have  investigated the long term behavior of solutions in the  (1-complex dimensional) strong unstable manifold of a nontrivial equilibrium $u_1$. 
By considering the equation \eqref{eq:CGL} for the parameters $ \theta \in \{ 0 , \pi/4, \pi /2\}$ we are able to investigate how our numerical technique behave in the dissipative regime, the dispersive regime, and in between.

By using rigorous numerics, we are able to answer questions that are not typically accessible by standard numerical methods. 
Namely  we are able to rigorously establish that the finite precision with which we performed our computations was sufficient to prove the existence of heteroclinic orbits from $u_1$ to $0$.  
Hence additional calculations at higher precision are not necessary for us to be confident of our result. 
And by using a forcing argument from complex analysis, we are able to establish the existence of unbounded solutions which limit to $u_1$ in negative time.

The restriction to only considering the strong unstable manifold of $u_1$ provides a tractable restriction for us to focus our analysis.   
This  isolates the primary difficulties of long integration time, and solutions which pass near a singularity.   
However as mentioned in Section \ref{sec:Introduction}, for $\theta \in (-\theta^*,\theta^*)$ the unstable manifold of $ u_1$ has complex dimension two. Thus our restriction to the strong unstable manifold is not capturing all of the dynamics in the entire unstable manifold.  
Not to mention that equation \eqref{eq:CGL} has infinitely other equilibria which we have not studied in this paper. 

Nevertheless, the entire picture of  how all of the dynamics behave on these strong unstable manifolds is still incomplete. 
Somewhat generalizing the $11^{th}$ open question in \cite{Jaquette2020}, we make the following conjecture. 
\begin{conjecture}
	For all $ \theta \in [ - \pi /2 , \pi / 2]$ there exists a unique trajectory in the (1-complex dimensional) strong unstable manifold of $ u_1$ which blows up in finite time. 
\end{conjecture}

It could be feasible to extend our current techniques used to prove Theorem \ref{prop:ManyHeteroclinics} by a parameter continuation in $\theta$, and prove that there is a heteroclinic orbit from $u_1$ to $0$ for all $\theta \in [-\pi/2,\pi/2]$. 
By the same argument in Theorem \ref{prop:GrowUp} this would force the existence of unbounded solutions which limit to $ u_1$ in backwards time, for all $ \theta \in (-\pi/2,\pi/2)$.  
However it still remains to prove the existence of unbounded solutions in the NLS case of $ \theta = \pm \pi /2$. 

Furthermore our current proof of unbounded solutions does not establish finite time blowup. 
In the case $ \theta = 0$ we proved in \cite{takayasu2019rigorous} that the initial data $ u_0(x) = 50(1-\cos( 2 \pi x))$ blows up at time $t_* \in  ( 0.0116, 0.0145)$ using rigorous numerics. 
This argument followed the approach of Masuda, wherein we integrated the solution in the complex plane of time to establish a branching singularity. 
Such an approach could be directly applied to prove blowup for other initial data, but only if the blowup set is robust and not isolated. 
In finite dimensional ODEs the recent work \cite{lessard2021geometric} has demonstrated computer-assisted proofs of unstable blowup. 
However it is unclear how to extend this technique to the infinite dimensional case.

\section*{Acknowledgements} 
The second author was supported by an NSERC Discovery Grant.
The third author was supported by JSPS KAKENHI Grant Numbers JP18K13453, JP20H01820, JP21H01001.

\bibliography{Bib_NLS,Bib_complexblowup}
\bibliographystyle{alpha}

\end{document}